\documentclass[a4paper,12pt]{amsart}
\usepackage[T1]{fontenc}
\usepackage[centering,margin=2.3cm]{geometry}
\usepackage[latin1]{inputenc}
\usepackage{amsthm, amsmath,url}   
\usepackage{latexsym,amssymb}
\usepackage{algorithmic}
\usepackage{amssymb}
\usepackage{amsaddr}
\usepackage{graphicx,subfigure}
\usepackage{times}
\usepackage{enumerate}
\usepackage[usenames,dvipsnames]{pstricks}
\usepackage{epsfig}
\usepackage{pst-node}
\usepackage{pst-grad} 
\usepackage{pst-plot} 
\usepackage{pst-3dplot}
\usepackage{color}
\usepackage{cite}
\newcommand{\N}{\mathbb{N}}

\newcommand{\Z}{\mathbb{Z}}

\newcommand{\End}{\mathrm{End}}
\newcommand{\Aut}{\mathrm{Aut}}
\newcommand{\Cay}{\mathrm{Cay}}

\theoremstyle{plain}
\newtheorem{theorem}{Theorem}[section]
\newtheorem{lemma}[theorem]{Lemma}
\newtheorem{corollary}[theorem]{Corollary}
\newtheorem{proposition}[theorem]{Proposition}

\theoremstyle{definition}

\newtheorem{conjecture}[theorem]{Conjecture}

\newtheorem{question}[theorem]{Question}
\newtheorem{remark}[theorem]{Remark}

\title[Beyond symmetry in Petersen graphs]{Beyond symmetry in generalized Petersen graphs}
\author[I. Garc\'{i}a-Marco]{Ignacio Garc\'{i}a-Marco $^*$}
\address{Facultad de Ciencias, Universidad de La Laguna, La Laguna, Spain}
\author[K. Knauer]{Kolja Knauer}
\address{Aix Marseille Univ, Universit\'e de Toulon, CNRS, LIS, Marseille, France\\Departament de Matem\`atiques i Inform\`atica,
Universitat de Barcelona, Barcelona, Spain}

\keywords{Generalized Petersen graph, endomorphism, retract, core, Cayley graph, monoid.  \\ \ \ \ $ ^*$ Corresponding author}


\begin{document}

\begin{abstract}
A graph is a \emph{core} or \emph{unretractive} if all its endomorphisms are automorphisms. Well-known examples of cores include the Petersen graph and the graph of the dodecahedron -- both generalized Petersen graphs. We characterize the generalized Petersen graphs that are cores. A simple characterization of endomorphism-transitive generalized Petersen graphs follows. This extends the characterization of vertex-transitive generalized Petersen graphs due to Frucht, Graver, and Watkins and solves a problem of Fan and Xie. 

Moreover, we study generalized Petersen graphs that are (underlying graphs of) Cayley graphs of monoids. We show that this is the case for the Petersen graph, answering a recent mathoverflow question, for the Desargues graphs, and for the dodecahedron -- answering a question of Knauer and Knauer. Moreover, we characterize the infinite family of generalized Petersen graphs that are Cayley graph of a monoid with generating connection set of size two. This extends Nedela and \v{S}koviera's characterization of generalized Petersen graphs that are group Cayley graphs and complements results of Hao, Gao, and Luo.
%
\end{abstract}

\maketitle

\section{Introduction}\label{introduction}

Let $k,n$ be integers such that $0<k<\frac{n}{2}$. The \emph{generalized Petersen graph} is the cubic graph $G(n,k)$ on vertex set $V = V_I \cup V_O$, being $V_I = \{v_0, \ldots, v_{n-1}\}$ the set of {\it inner vertices} and $V_O = \{u_0,\ldots, u_{n-1}\}$ the set of {\it outer vertices}. The edge set is partitioned into three parts (all subscripts are considered modulo $n$): the edges $E_O(n,k) = \{u_i u_{i+1} \, \vert \, 0 \leq i \leq n-1\}$ form the {\it outer rim}, inducing a cycle of length $n$; the edges $E_I(n,k) = \{v_i v_{i+k} \, \vert \, 0 \leq i \leq n-1\}$ form the {\it inner rims}, inducing $\gcd(n,k)$ cycles of length $n/\gcd(n,k)$; and the edges $E_S(n,k) = \{u_i v_i \, \vert \, 0 \leq i \leq n-1\}$ called {\it spokes} forming a perfect matching of $G(n,k)$. Generalized Petersen graphs were introduced by Coxeter~\cite{Cox-50} and named by Mark Watkins~\cite{Watkins:69}.
Many known cubic graphs belong to this class, e.g., the \emph{Petersen graph} $G(5,2)$ itself, the \emph{D\"urer graph} $G(6,2)$, the \emph{M\"obius-Kantor graph} $G(8,3)$ (Figure~\ref{fig:G83}), the \emph{dodecahedron} $G(10,2)$, the \emph{Desargues graph} $G(10,3)$, the \emph{Nauru graph} $G(12,5)$, and the \emph{$n$-prism} $G(n,1)$. Coxeter even wrote a paper on $G(24,5)$, see~\cite{COXETER1986579}.
Despite its simple definition, many important algebraic properties of $G(n,k)$ depend on the particular $k,n$, e.g., isomorphisms~\cite{Ste-09}, automorphism groups, edge-and vertex-transitivity~\cite{FGW:71}, being Cayley graph of a group~\cite{Ned-95,Lov-97}. 

\begin{figure}
\includegraphics[width=.4\textwidth]{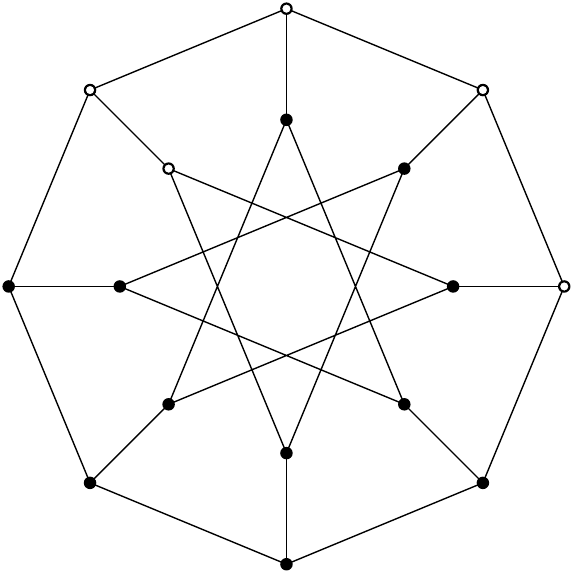}
\caption{The {M\"obius-Kantor graph} $G(8,3)$.} \label{fig:G83}
\end{figure}

In the present paper, we study what is sometimes called ``generalized symmetries'' of generalized Petersen graphs, see e.g.~\cite{Fan-09}. This is, we study endomorphisms and retracts of $G(n,k)$, as well as Cayley graphs of semigroups and monoids that are generalized Petersen graphs.
Graph homomorphisms and in particular the structure of the endomorphism monoid of a graph are classical topics of research, with several books dedicated or closely related to them, see e.g.~\cite{Kna-19,Hel-04,God-01}. This type of questions concern the first part of the present paper. In particular, we characterize cores in generalized Petersen graphs (Theorem~\ref{thm:cores}). As a corollary we obtain the characterization of endomorphism-transitive generalized Petersen graphs (Corollary~\ref{cor:endtrans}). This settles a problem of Fan and Xie~\cite{Fan-04,Fan-09}. This can be seen as an extension of the classical characterization of vertex-transitive generalized Petersen graphs of Frucht, Graver, and Watkins~\cite{FGW:71}.

The second part of the paper is dedicated to Cayley graphs of monoids and semigroups. These form a more complicated class than their group counterpart and are related to regular languages in automata theory~\cite{HAT} and have applications in Data-Mining~\cite{KRY2009}. An important theoretical feature of Cayley graphs concerns the representation theory of monoids as endomorphism monoids of graphs, see~\cite{HL69,HP64,HP65} -- an area with recent~\cite[Problem 19.2]{NOdM12} and old questions~\cite{BP80}.
In the study of Cayley graphs of semigroups two main directions can be identified. 
On the one hand properties of Cayley graphs of special classes of semigroups have been investigated, see~\cite{zbMATH05610940,Kel-06,zbMATH06864652,zbMATH06740692,zbMATH06948232,zbMATH06613956,zbMATH06184562,zbMATH06147894,zbMATH06029728,zbMATH06120589,zbMATH06093204,zbMATH05973394,zbMATH05886836,zbMATH06139402,zbMATH05701935,zbMATH05812671}. On the other hand, Cayley graphs falling into a certain class of graphs have been studied, such as acyclic, transitive digraphs~\cite{garcamarco2019cayley}, bidirected digraphs~\cite{Kel-02}, transitive digraphs~\cite{Kel-03}, and bounded outdegree digraphs~\cite{Kna-21,Zelinka1981}. 
Semigroups that admit a generating set such that the Cayley graph has given genus have been studied~\cite{Sol-06,Sol-11,Zha-08,Kna-10,Kna-16}.  
In the topological setting edge orientations, multiplicities, and loops can be ignored. This leads to simple undirected underlying graphs of Cayley graphs -- a notion that in contrast to the group setting causes a significant loss of algebraic information. Only recently graphs that are not the underlying graph of Cayley graphs of monoids have been found~\cite{Kna-21}. The main question of the second part of the paper is:
\begin{quote}
Which generalized Petersen graphs are underlying graphs of Cayley graphs?
\end{quote}

First, as a corollary of our study of cores, we show that there are infinitely many generalized Petersen graphs which cannot be the underlying graph of a loopless Cayley graph (Corollary~\ref{cor:petnotcayley}). This answers a question of~\cite[Question 6.6]{garcamarco2019cayley} and strengthens a result of~\cite{Kho-21} for monoids (Corollary~\ref{cor:manybad}).
Moving on to general Cayley graphs (with possible loops), we 
answer the recent question on \emph{mathoverflow} whether the Petersen graphs was Cayley graph of a group-like structure~\cite{390161}. We present four different ways to represent the Petersen graph as a Cayley graph (Proposition~\ref{prop:Petersen}). Furthermore, we show that the Kronecker cover of the Petersen graph -- the Desargues graph is the underlying graph of a monoid Cayley graph (Proposition~\ref{prop:Desargues}).
The planar connected Cayley graphs of groups are exactly the graphs of the Platonic and Archimedean solids except the Dodecahedron and the Icosidodecahedron~\cite{Mas-96}. This led to the question whether the latter two are underlying graphs of Cayley graphs of semigroups or monoids, see~\cite[Problem 4]{Kna-16}. We answer this question partially by providing a monoid representation of the dodecahedron (Proposition~\ref{pr:dodeca}). 
Finally, we characterize those generalized Petersen graphs that are Cayley graphs of a monoid with respect to a generating connection set of size two (Theorem~\ref{thm:mon2gen}). This extends Nedela and \v{S}koviera's~\cite{Ned-95} characterization of generalized Petersen graphs that are group Cayley graphs as well as results by Hao, Gao, and Luo~\cite{zbMATH05973394,zbMATH06093204} about generalized Petersen graph as components of Cayley graphs of symmetric inverse and Brandt semigroups.

\section{Cores and endomorphism-transitivity}
In this section we present a
characterization of the unretractive generalized Petersen graphs. As a corollary we characterize of endomorphism-transitive generalized Petersen graphs -- settling a problem of Fan and Xie~\cite{Fan-04,Fan-09}.

Unless the graph is just an edge or a vertex, bipartite graphs are not cores. Moreover, bipartite graphs without isolated vertices are endomorphism-transitive. For this reason in this section we will only consider non bipartite graphs. It is easy to check that the generalized Petersen graph $G(n,k)$ is bipartite if and only if $n$ is even and $k$ is odd.

\begin{theorem}\label{thm:cores} Let $G(n,k)$ be a non-bipartite generalized Petersen graph.  Then, the following conditions are equivalent:
\begin{itemize}
\item[(a)] $G(n,k)$ is a core,
\item[(b)] one of the minimum odd length cycles of $G(n,k)$ uses a spoke,
\item[(c)] If we denote by $d := \gcd(n,k)$
and by $a \in \Z^+$ the only integer $0 < a < n/d$ such that $ak \equiv d \ ({\rm mod}\ n)$, then one of the following properties holds: 
\begin{itemize} \item[(c.1)] $n/d$ is even, or 
\item[(c.2)] $a + d$ is even and $a \geq d + 2$, or
\item[(c.3)] $a + d$ is odd and $a + d + 2 \leq n/d$.
\end{itemize}
\end{itemize}
\end{theorem}

As a consequence we derive the following characterization of endomorphism transitive generalized Petersen graphs.

\begin{corollary}\label{cor:endtrans}
 The endomorphism transitive generalized Petersen graphs are exactly the transitive and the bipartite generalized Petersen graphs.
\end{corollary}

Before proving the main results of this section, we begin by summarizing the main results that we need about cores (see, e.g., ~\cite[Section 6.2]{God-01} for a proof of these statements) and about the automorphism group of generalized Petersen graphs. A graph $G$ is a \emph{core} (or \emph{unretractive}) if all its endomorphisms are automorphisms. A subgraph $X$ of $G$ is {\it a core of $G$} if $X$ is a core itself and there is a homomorphism from $G$ to $X$.  A {\it retraction} is a homomorphism $f$ from $G$ to a subgraph $X$ such that $f$ restricted to $V(X)$ is the identity map. When there is a retraction from $G$ to $X$ we say that $X$ is a {\it retract} of $G$. 
Clearly retracts have to be induced subgraphs and retracts of connected graphs have to be connected too. We call the \emph{odd girth} of a non-bipartite $G$ the length of a shortest odd cycle. Since homomorphic images of odd cycles are odd cycles, if $G$ is non bipartite with odd girth $g$ and $X$ is a retract of $G$, then there is a cycle of length $g$ in $X$.

\begin{proposition}\label{prop:cuore} Every graph $G$ has an (up to isomorphism) unique core $X$. Moreover, $X$ is a retract of $G$. As a consequence, if $X$ is a core of $G$ and there exists $\varphi \in {\rm Aut}(G)$ such that $\varphi(x) = y$ for some $x,y \in V(X)$, then there exists $\phi \in {\rm Aut}(X)$ such that $\phi(x) = y$.
\end{proposition}

The automorphism group of $G(n,k)$ depends on the values of $n,k$ and was completely described by Frucht, Graver and Watkins in~\cite{FGW:71}. Let $V$ be the vertex set of $G(n,k)$, and consider $\alpha,\beta,\gamma: V \longrightarrow V$ defined as 
\begin{eqnarray}
\label{eq:alpha} \alpha(u_i) = u_{i+1}, &  \alpha(v_i) = v_{i+1}, &\hspace{1cm} \text{(rotation),} \\ 
\label{eq:beta} \beta(u_i) = u_{-i}, & \beta(v_i) = v_{-i},&\hspace{1cm} \text{(reflection),}  \\
\label{eq:gamma} \gamma(u_i) = v_{ki}, & \gamma(v_i) = u_{ki},&\hspace{1cm} \text{(inside-out).}
\end{eqnarray}  Then $\alpha, \beta$ are always automorphisms of $G(n,k)$ and, in particular, ${\rm Aut}(G(n,k))$ has a dihedral subgroup $D_n$ or order $2n$. As a consequence, we have that $G(n,k)$ can be either transitive or have two orbits, $V_I$ and $V_O$. Moreover, $\gamma$ is automorphism if and only if $k^2 \equiv \pm 1 \ ({\rm mod}\ n)$. Except in seven exceptional cases (which are described in detail in~\cite{FGW:71}), the group ${\rm Aut}(G(n,k))$ can be described with $\alpha, \beta$ and $\gamma$.
 
\begin{theorem}\label{thm:autgroup}~\cite{FGW:71} If $(n,k)$ is not one of $(4,1), (5,2), (8,3), (10,2), (10,3), (12,5)$ or $(24,5)$, then the following hold:
\begin{itemize}
\item if $k^2 \equiv 1 \ ({\rm mod}\ n)$; then \[ {\rm Aut}(G(n,k)) = \langle \alpha, \beta, \gamma \, \vert \, \alpha^n = \beta^2 = \gamma^2 = {\rm id},\, \beta \alpha \beta = \alpha^{-1},\, \gamma \beta = \beta \gamma,\, \gamma \alpha \gamma = \alpha^k \rangle, \]
\item if $k^2 \equiv -1 \ ({\rm mod}\ n)$; then $ {\rm Aut}(G(n,k)) = \langle \alpha,  \gamma \, \vert \, \alpha^n =  \gamma^4 = {\rm id},\, \gamma \alpha \gamma^{-1} = \alpha^k \rangle,$
\item if $k^2 \not\equiv \pm 1 \ ({\rm mod}\ n)$; then $ {\rm Aut}(G(n,k)) = \langle \alpha,  \beta \, \vert \, \alpha^n =  \beta^2 = {\rm id},\, \beta \alpha \beta  = \alpha^{-1}\rangle.$
\end{itemize}
\end{theorem}

As a consequence of this, in~\cite{FGW:71}, they also get the following.

\begin{corollary}\label{cor:autdi} The following are equivalent:
\begin{itemize} \item[(a)] $G(n,k)$ is transitive,
\item[(b)] $k^2 \equiv \pm 1 \ ({\rm mod}\ n)$ or $(n,k) = (10,2)$ (dodecahedron), 
\item[(c)] the automorphism group of $G(n,k)$ is different from the dihedral group $D_n$. 
\end{itemize}
\end{corollary}

We will also use the following auxiliary results to prove that (b) implies (a) in Theorem~\ref{thm:cores}.

\begin{lemma}\label{lem:genprismnocore}
Let $G$ be the graph consisting of two vertex disjoint cycles of the same length $C_1 = (x_0,\ldots,x_{\ell-1})$, $C_2 = (y_0,\ldots,y_{\ell-1})$ and $\ell$ disjoint paths of the same length $P_0,\ldots, P_{\ell-1}$, where $P_i$ joins $x_i$ with $y_i$ for all $i \in \{0,\ldots,\ell-1\}$. Then, $G$ is not a core.
\end{lemma}
\begin{proof}Denote by $k$ the length of $P_0,\ldots, P_{\ell-1}$ and $P_i = (x_i = z_{i,0}, z_{i,1},\ldots,z_{i,k} = y_i)$ for all $i \in \{0,\ldots,\ell-1\}$. Now, we define $\varphi \in {\rm End}(G)$ as $\varphi(z_{i,j}) = x_{i+j\ {\rm mod} \ \ell}\ $ for all $i \in \{0,\ldots,\ell-1\}$. We have that $\varphi$ is not an automorphism and, thus, $G$ is not a core. 
\end{proof}

\begin{lemma}\label{lem:oddgirthinteriorcycles}
Let $G(n,k)$ be a non bipartite generalized Petersen graph of odd girth $g$. Then, every odd cycle of length $g$ has at most two spokes. 
\end{lemma}
\begin{proof}Let $C$ be a odd cycle of length $\ell$ using $a$ exterior edges, $b$ interior edges and $c$ spokes with $c > 2$, and let us prove that $\ell = a + b + c > g$.  Since the set of spokes is a cut-set, then $c$ is even. One can construct an odd closed walk of length $a+b+2$ using $a$ exterior edges, $b$ interior edges all belonging to the same inner rim, and $2$ spokes.  Hence, $\ell = a + b + c > a + b + 2 \geq g$.    
\end{proof}

Now, we can prove that (b) implies (a) in Theorem~\ref{thm:cores}

\begin{proposition}\label{pr:bimpliesa}
Let $G(n,k)$ be a non bipartite generalized Petersen graph of odd girth $g$. If there exists a cycle $C$ of length $g$ passing through both inner and outer vertices, then $G(n,k)$ is a core.
\end{proposition}
\begin{proof}We first observe that if $G(n,1)$ is non bipartite, then the odd girth is $n$ and there are no cycles of length $n$ passing through inner and outer vertices. Take now $G(n,k)$ a non bipartite generalized Petersen graph having a cycle $C$ of length $g$ passing through both inner and outer vertices. By our previous observation we have that $k \geq 2$. 

 Let $C_g = (w_0,\ldots,w_{g-1},w_0)$ be a cycle of length $g$ and let us prove that there are no homomorphisms from  $G(n,k)$ to $C_g$. Assume by contradiction that there is such a homomorphism~$\phi$. Consider the  $8$-cycle  $C_8 = (u_0,u_1,v_1, v_{k+1}, u_{k+1}, u_{k}, v_k, v_0, u_0)$ in $G(n,k)$.  One observes that in every two consecutive edges in $C_8$, one of them is a spoke (see Figure~\ref{fig:10-2}). Thus, every three consecutive vertices in $C_8$ belong to a cycle of length $g$ (this is because $G(n,k)$ has an odd cycle of length $g$ with inner, outer and spoke edges). 
Assume without loss of generality that $\phi(u_0) = w_0$ and $\phi(u_1) = w_1$. Since $v_1$ is adjacent to $u_1$, then $\phi(v_1) \in \{w_0,w_2\}$. Moreover, we have that $u_0,u_1,v_1$ belong to a cycle $C$ of length $g$, hence $\phi$ restricted to the vertices of $C$ has to be bijective and, thus $\phi(v_1) = w_2$. Repeating an analogous argument we get that $\phi(v_{k+1}) = w_{3},\, \phi(u_{k+1}) = w_{4},\ldots,\, \phi(u_0) = w_{8},$ where the subindices in the vertices of $C_g$ are taken modulo $g$. Thus, we get that $w_0 = \phi(u_0) = w_{8}$ and $g$ divides $8$, a contradiction. 

\begin{figure}
\includegraphics[scale=.7]{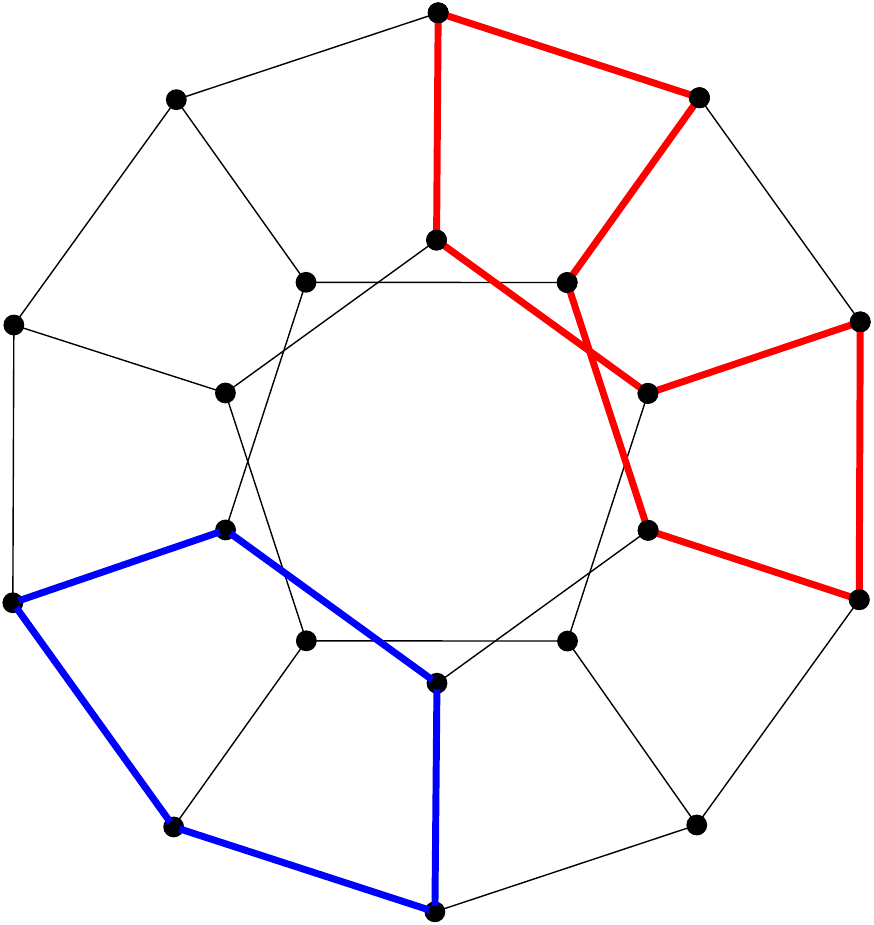}
\caption{The dodecahedron graph $G(10,2)$ has odd girth equal to $5$ and there exists a $5$-cycle (in blue) passing through both inner and outer vertices. As a consequence, every three consecutive vertices of the $8$-cycle $C_8$ (in red) belong to a $5$-cycle.} \label{fig:10-2}
\end{figure}

Let $X$ be a core of $G(n,k)$. We know that $X$ contains a cycle of length $g$ but we have proved that $X$ itself is not a cycle of length $g$. 
Since $X$ has to be connected, then $X$ has both inner and outer vertices of $G(n,k)$.  By  Proposition~\ref{prop:cuore}, we also have that all inner (respect. outer) vertices have the same degree in $X$, which we denote $d_I$ (respect. $d_O$). We have that $d_I \geq 2$ and $d_O \geq 2$. Moreover, since $X$ is connected and is not a cycle, then $d_I$ and $d_O$ cannot be both two. If $d_O = 3$, since $X$ is an induced subgraph,  we have that $X = G(n,k)$. Assume now that $d_I = 3$ and $d_O = 2$ and let us prove that this cannot happen. We denote by $t$ the number of inner cycles of $X$ (that is, the number of connected components of the induced subgraph with vertices $V(X) \cap V_I$).  Since the number of vertices of odd degree in $X$ has to be even, then either $t$ is even or the inner cycles are even.  We separate two cases:

\emph{Case 1: $t$ is even:}  If $t = 2$, then $X$ is the graph in Lemma~\ref{lem:genprismnocore} and, thus, it is not a core (see Figure~\ref{fig:subPet15-3}). If $t \geq 4$, then $X$ is not connected, which is not possible.

\begin{figure}
\includegraphics[scale=.7]{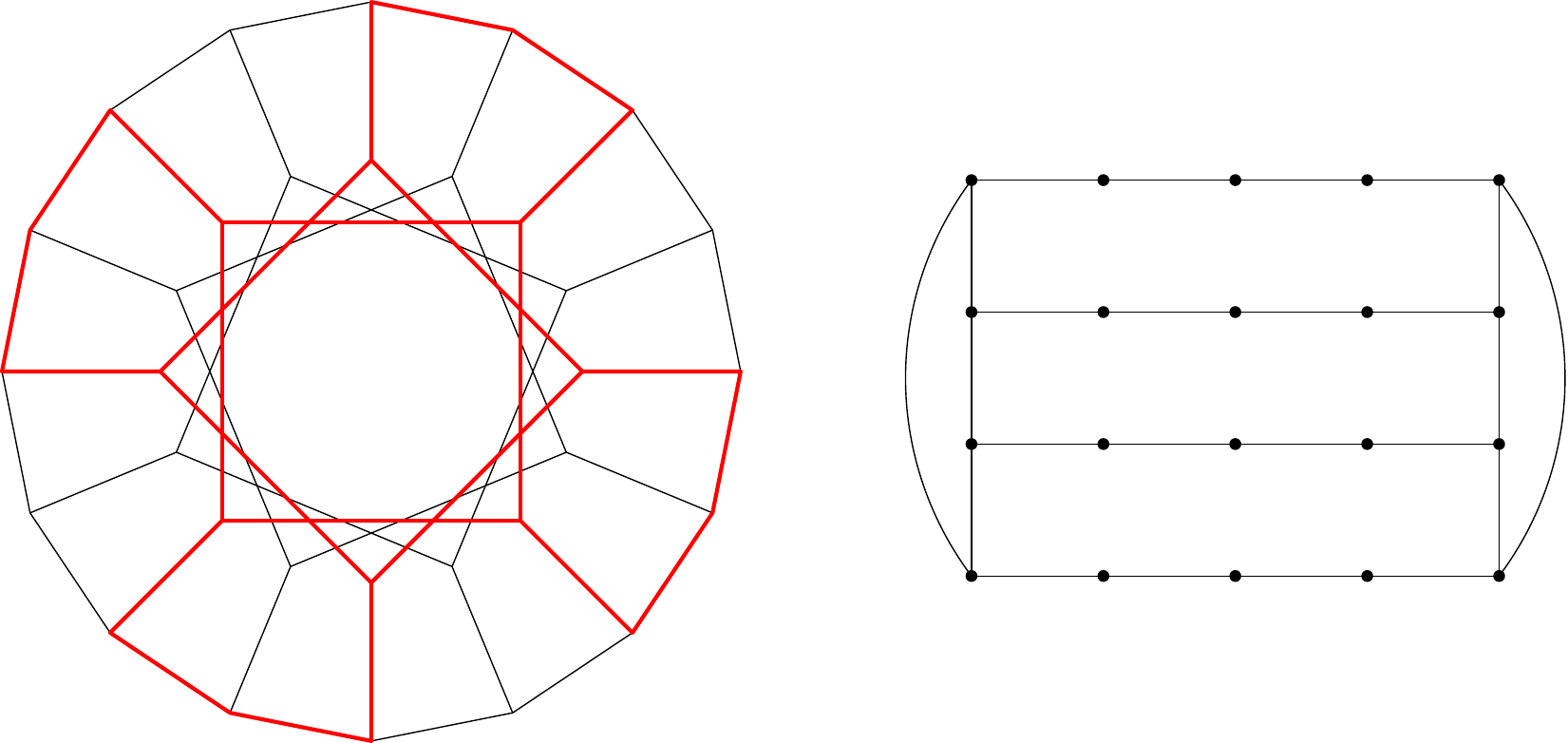}
\caption{Left: graph $G(16,4)$ with a connected induced subgraph $X$ (in red). In $X$ all interior vertices have degree $3$, all exterior vertices have degree $2$, and $X$ contains $2$ inner cycles. Right: a graph isomorphic to $X$.} \label{fig:subPet15-3}
\end{figure}

\emph{Case 2: $t$ is odd and the inner cycles are even:}  If $t \geq 3$, we claim that $X$ has no cycles of length $g$. Indeed, if $X$ has a cycle of length $g$, then it uses at most two spoke edges by Lemma~\ref{lem:oddgirthinteriorcycles}. However, all the induced subgraphs of $X$ with at most two spoke edges are bipartite, a contradiction (see Figure~\ref{fig:Pet24-6}).

\begin{figure}
\includegraphics[scale=.7]{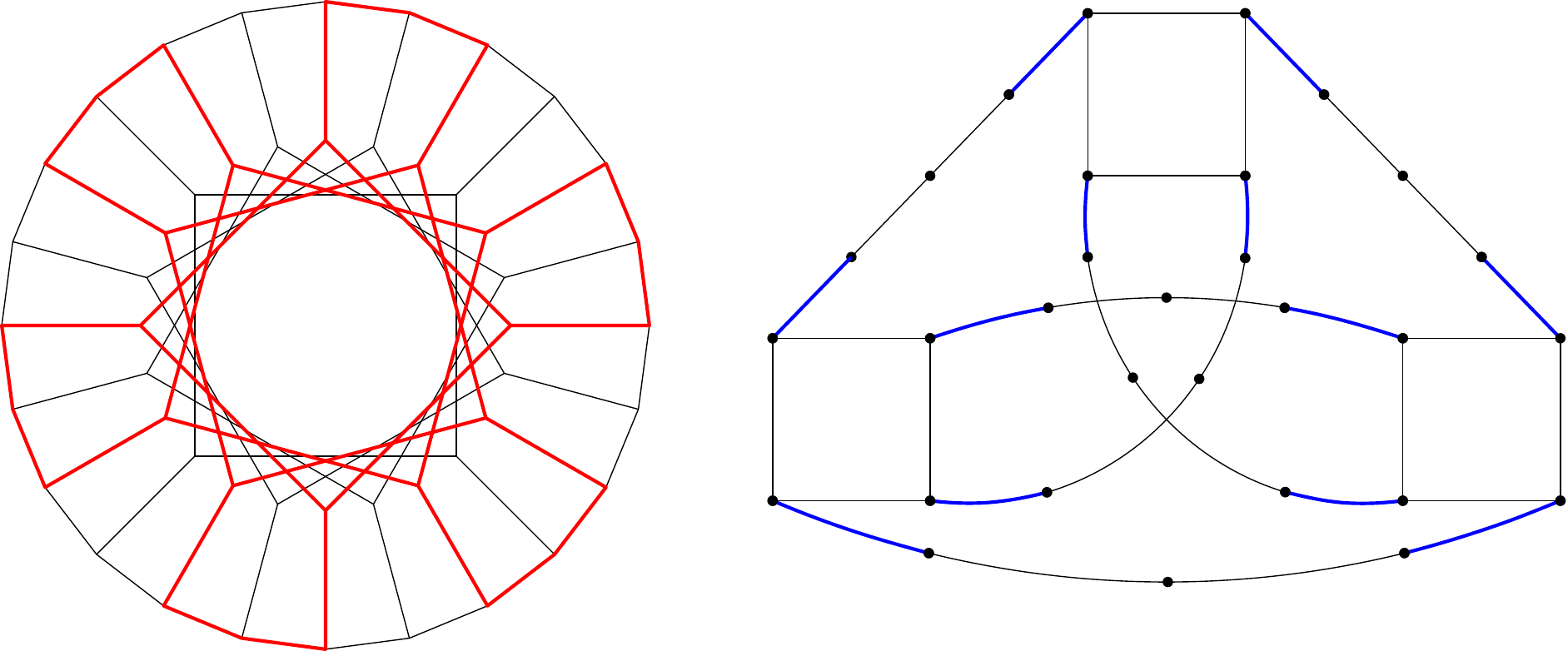}
\caption{Left: graph $G(24,6)$ with a connected induced subgraph $X$ in red. In $X$ all interior vertices have degree $3$, all exterior vertices have degree $2$, and $X$ contains $3$ inner cycles. Right: a graph isomorphic to $X$, where the spoke edges are in blue. Any subgraph of $X$ with at most two blue edges is bipartite.} \label{fig:Pet24-6}
\end{figure}

If $t = 1$, then denoting $d := \gcd(n,k)$ we have that $X$ is isomorphic to the induced graph $Y$ with inner vertices $\{ v_{i} \, \vert \, d$ divides $i\}$ and outer vertices $\{u_i \, \vert \, d$ divides $i\} \cup \{u_i \, \vert \, \lfloor i/d \rfloor$ is even$\}$ (see Figure~\ref{fig:subPet16-6}). We have that $Y$ has a cycle of length $g$, and the inner cycles are all even. Then, by Lemma~\ref{lem:oddgirthinteriorcycles}, every cycle of length $g$ in $Y$ uses exactly two spokes. Therefore, there is a cycle $C$ of length $g$ such that $V(C) \cap V_O = \{u_0,u_1,\ldots,u_d\}$ and $u_0v_0,u_dv_d \in E(C)$.  Similarly, there is a cycle $C'$ of length $g$ such that $V(C') \cap V_O = \{u_d,u_{d+1},\ldots,u_{2d}\}$ and $u_dv_d, u_{2d}v_{2d} \in E(C')$. Now, consider $\pi$ the retraction from $G(n,k)$ to $Y$. Since $\pi$ is the identity on $Y$ we have that $\pi(u_d) = u_d,\, \pi(v_d) = v_d, \, \pi(u_{2d}) = u_{2d},\,  \pi(v_{2d}) = v_{2d}$. Moreover, $\pi$ sends $C'$ to a cycle of length $g$ in $Y$. However, any cycle in $Y$ containing the edges $u_dv_d,u_{2d}v_{2d}$ passes through (at least) four spokes, and this contradicts Lemma~\ref{lem:oddgirthinteriorcycles}.

\begin{figure}
\includegraphics[scale=.7]{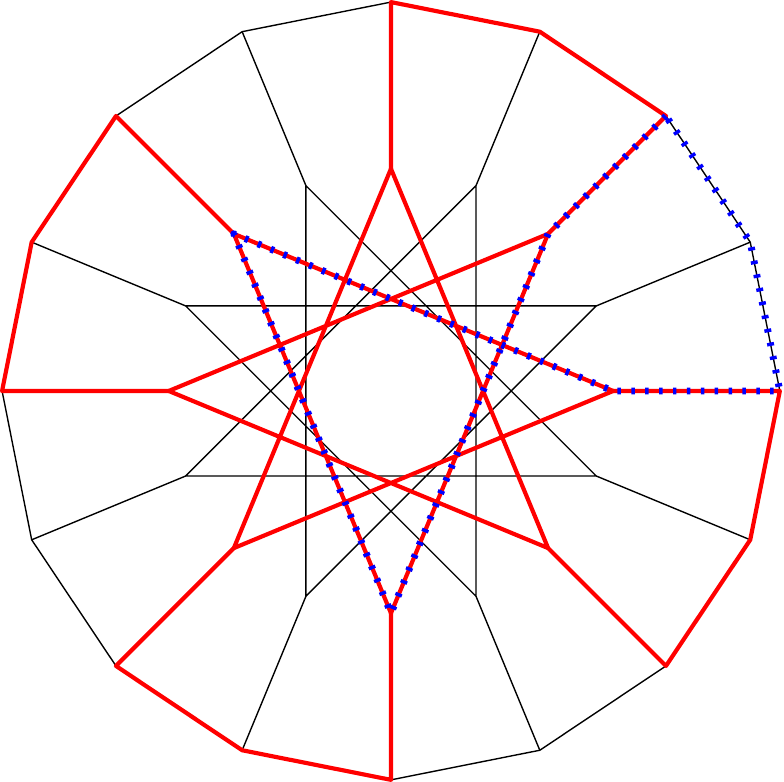}
\caption{Graph $G(16,6)$ with a connected induced subgraph $Y$ in red. In $Y$ all interior vertices have degree $3$, all exterior vertices have degree $2$, and $Y$ contains one inner cycle. In blue, the $7$-cycle $C'$ whose length equals the odd girth of $G(16,6)$.} \label{fig:subPet16-6}
\end{figure}

\end{proof}

Now we can proceed with the proof of the main result.

\noindent {\bf Proof of Theorem~\ref{thm:cores}.} 
$(b) \Longrightarrow (a)$ is Proposition~\ref{pr:bimpliesa}. Let us prove $(c) \Longrightarrow (b)$. We observe that the subgraph induced by $V_I$ consists of $d$ disjoint cycles of length $n/d$ and the subgraph induced by $V_0$ is just a cycle of $n$ vertices. 
If $(c.1)$ holds, then $n/d$ (and $n$) are even, so any odd cycle has inner and outer vertices. So, assume that $n/d$ is odd and if either $(c.2)$ or $(c.3)$ holds, we are going to find an odd cycle of length $\leq n/d$ passing through inner and outer vertices.    If $(c.2)$ holds, the odd cycle 
\[ (v_d = v_{ak}, v_{(a+1)k}, \ldots, v_{(n/d) k} = v_0, u_0, u_{1},\ldots, u_d, v_d).\]
has length $n/d - a + d + 2 \leq n/d$, so we are done. If $(c.3)$ holds, the odd cycle
\[ (v_0, v_k, \ldots, v_{ak} = v_d, u_d, u_{d-1},\ldots, u_0, v_0).\] 
has length $a + d + 2 \leq n/d$, so we are done.

Let us prove now $(a) \Longrightarrow (c)$ by contradiction.  So, we assume that $n/d$ is odd, and 
\begin{itemize} \item[(1)] if $a+d$ is even, then $a \leq d$ and \item[(2)] if $a+d$ is odd, then $a+d \geq n/d$,
\end{itemize} and we aim at proving that $G(n,k)$ is not a core in either case. For this purpose, we are going to describe a retraction $f$ from $G(n,k)$ to one of the inner cycles, namely  \[ C = (v_0, v_k, v_{2k}, \ldots, v_{(g-1)k}, v_{gk} = v_0).\] Since  $a k \equiv d \ ({\rm mod}\ n)$,
we have that $v_d \in V(C)$. We separate two cases:

If (1) holds, we define $f$ for the vertices in the outer rim as follows: for $i \in \Z$ we denote by $q$ and $r$ the quotient and remainder of the Euclidean division of $i$ by $d$, i.e., $i = q d + r$ with $0 \leq r < d$, and set

\begin{equation}\label{eq:defretracto}  f(u_i) = \left\{ \begin{array}{lll} v_{qd + (r+1)k} & {\rm  if} & r = 0,\ldots, a-1,  \\    v_{(q+1)d+k} & {\rm  if } & a \leq r < d,\, r \equiv a\ ({\rm mod}\ 2),\\  v_{(q+1)d} & {\rm  if } & a < r < d,\, r \not\equiv a\ ({\rm mod}\ 2). \end{array} \right. \end{equation}
Concerning the inner rim: we set  $f(v_j) = v_{l - k}$ whenever  $f(u_j) = v_l$,  for all $j \in \N$; as usual, all the subindices are taken modulo $n$ (see Figures~\ref{fig:15-3} and~\ref{fig:15-6}).

\begin{figure}
\includegraphics[scale=.5]{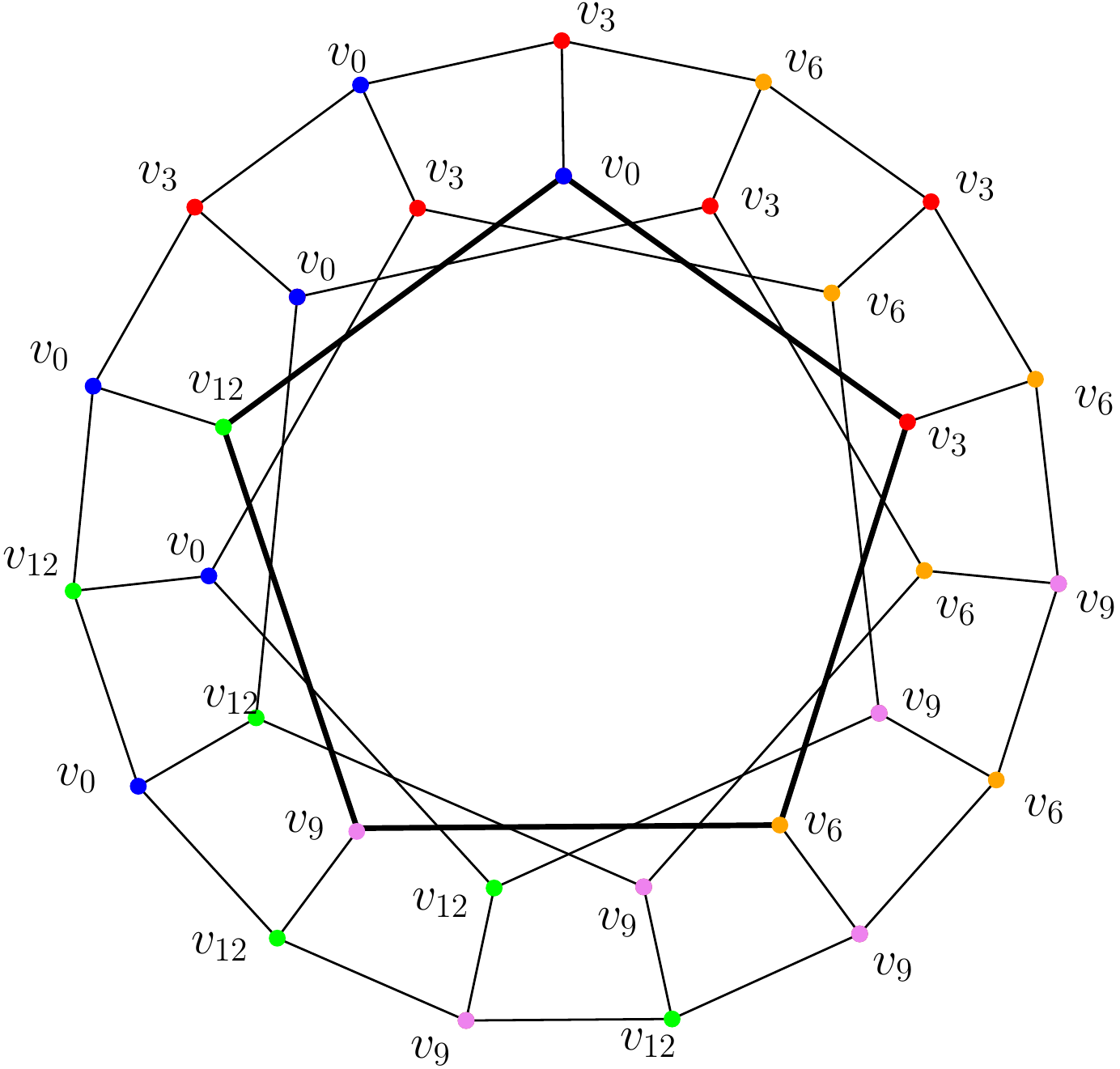}
\caption{Figure illustrating the retraction from $G(15,3)$ to the bold 5-cycle $C = (v_0,v_3,v_6,v_9,v_{12},v_0)$ described in (\ref{eq:defretracto}). In this example $g = 5$, $d = 3$ and $a = 1$.} \label{fig:15-3}
\end{figure}
\begin{figure}
\includegraphics[scale=.5]{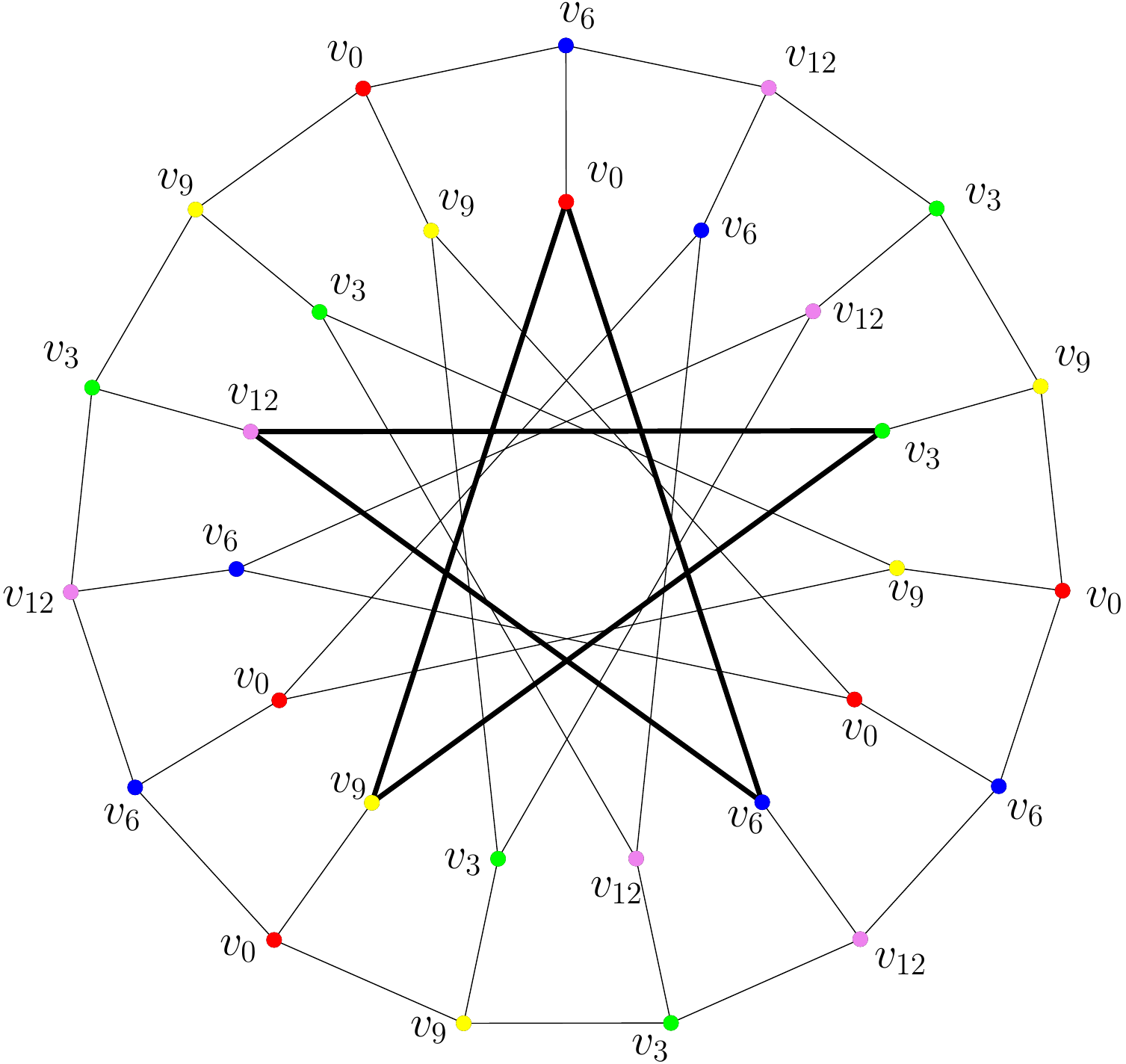}
\caption{Figure illustrating the retraction from $G(15,6)$ to the bold 5-cycle $C = (v_0,v_6,v_{12},v_3,v_{9},v_0)$  described in (\ref{eq:defretracto}). In this example $g = 5$, $d = 3$ and $a = 3$.}\label{fig:15-6}
\end{figure}

{\it $f$ is well-defined}.  This holds since $f(u_{i+n}) = f(u_i)$ and $f(v_{i+n}) = f(v_i)$ for all $i \in \Z$..

{\it $f$ is a homomorphism.} That is, if $xy \in E(G(n,k))$, then $f(x) f(y) \in E(G(n,k))$, we separate three cases:
if $xy \in E_S(n,k)$ is a spoke, then $f(x)f(y) \in E(G(n,k))$ by definition. If $u_i u_{i+1} \in E_O(n,k)$ is an outer edge with $0 \leq i \leq n-1$. We denote by $q$ (respect. $q'$) and $r$ (respect. $r'$) the quotient and the remainder of the division of $i$ (respect. $i+1$) by $d$. 
\begin{itemize}
\item Case $r < d$: we have that $q = q'$, $r' = r+1$. Thus, $f(u_i) f(u_{i+1}) \in E(G(n,k))$. 
\item Case $r = d-1$: we have that $q' = q+1$, $r' = 0$ and, since $r \not\equiv a \ ({\rm mod}\ 2)$,  it follows that $f(u_i) = v_{(q+1)d}$ and $f(u_{i+1}) = v_{q'd + k} = v_{(q+1)d+k}$. Thus, $f(u_i) f(u_{i+1}) \in E(G(n,k))$.  
\end{itemize}

Inner edges: let $v_i v_{i+k} \in E_I(n,k)$. We observe that $f(v_i) f(v_{i+k}) \in E(G(n,k))$ if and only if $f(u_i) f(u_{i+k}) \in E(G(n,k))$. Moreover, denoting by $q$ (respect. $q'$) and $r$ (respect. $r'$) the quotient and the remainder of the division of $i$ (respect. $i+k$) by $d$; we have that $q' = q + (k/d)$ and $r' = r$. It follows that if $f(u_i) = v_l$ then $f(u_{i+k}) = v_{l + (k/d)d} = v_{l+k}$. Thus $f(u_i) f(u_{i+k}) \in E(G(n,k))$.

{\it $f$ is a retraction from $G(n,k)$ to $C$.} It is clear that the image of every vertex lies in $V(C)$, so it just remains to prove that $f(v) = v$ for all $v \in V(C)$. Take $v \in V(C)$, then $v = v_{\lambda k}$ for some $\lambda \in \{0,\ldots,g - 1\}$.  Following (\ref{eq:defretracto}) we have that $f(u_{\lambda k}) = v_{\lambda k + k}$ and, thus, $f(u_{\lambda k}) = u_{\lambda k}$; and we are done.

If (2) holds, we consider $G(n,k') = G(n, n-k)$ one has that for this new graph $a' = g - a$ and, thus $a' + d$ even, so the condition $a + d \geq n/d$ is equivalent to $a' \leq d$ and we can apply case (1) to $G(n,k')$ to get the result. 
Note that in this part of the proof we are considering a nonstandard generalized Petersen graph $G(n,k')$, where $k' \geq n/2$ (see Figure~\ref{fig:10-4} for an example of the retraction~$f$ obtained in this way). $\hfill \qed$

\begin{figure}

\includegraphics[scale=.7]{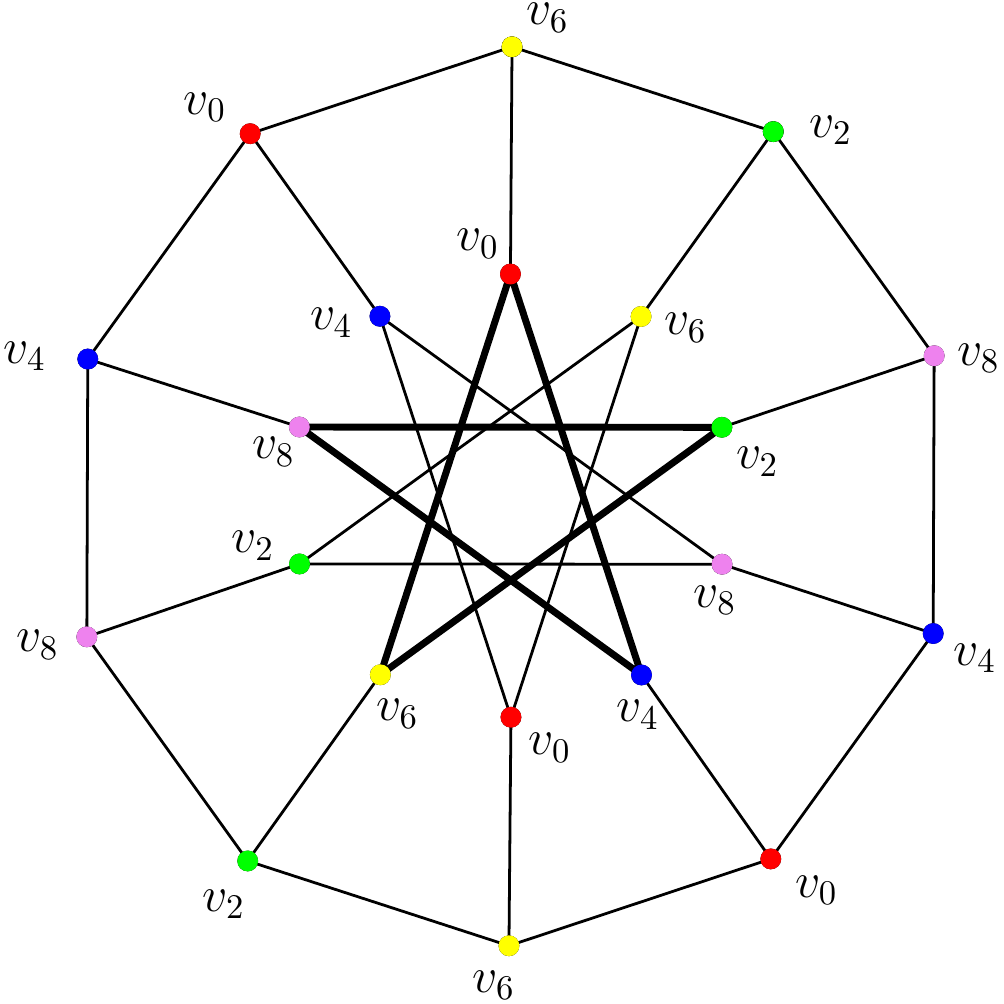}
\caption{Figure illustrating the retraction from $G(10,4)$ to the bold 5-cycle $C = (v_0,v_4,v_{8},v_2,v_{6},v_0)$. In this example $g = 5$, $d = 2$ and $a = 3$.}\label{fig:10-4}
\end{figure}

As an  easy consequence we have the  following result:

\begin{corollary}\label{cor:primos} Let $G(n,k)$ be a generalized Petersen graph. 
If $\gcd(n,k) = 1$, then $G(n,k)$ is a core if and only if $G(n,k)$ is not bipartite and $k \neq 1$, i.e., $G(n,k)$ is not the $n$-prism.
\end{corollary}
\begin{proof} It is easy to check that $n$-prisms are not cores (see, e.g., Theorem~\ref{thm:cores}) and that bipartite (nontrivial) graphs are not cores either. 

Let $G(n,k)$ be a non bipartite generalized Petersen graph with $k \neq 1$ and $\gcd(n,k) = 1$. Then, both the inner and the outer rims are cycles of length $n$. 

If $n$ is even, then all odd cycles pass through inner and outer vertices. Thus,  $G(n,k)$ is  a core by Theorem~\ref{thm:cores}.(b). 
Assume now that $n$ is odd, we separate two cases. If $k$ is even, we consider the odd cycle $C_1 = (u_0, u_1,\ldots, u_k, v_k, v_0, u_0)$ of length $k + 3 \leq 2k + 1 \leq n$. If $k$ is odd (and $k \neq 1$), we consider the odd cycle $C_2 = (u_0, u_{n-1},\ldots,u_k, v_k, v_0, u_0)$ of length $n - k + 3 \leq n$.  In both cases we have found an odd cycle of length $\leq n$ passing through inner and outer vertices. Hence, the result follows from Theorem~\ref{thm:cores}.(b). 
\end{proof}

\bigskip

\noindent {\bf Proof of Corollary~\ref{cor:endtrans}}. Transitive and bipartite graphs are always endomorphism transitive and core graphs are transitive if and only if they are endomorphism transitive. So, to finish the proof it suffices to consider $G(n,k)$ non bipartite, not a core and endomorphism transitive and prove that it is transitive. Since $G(n,k)$ is not a core, by Theorem~\ref{thm:cores}.(b) we have that there is no odd cycle of length $g$ passing through inner and outer vertices. As a consequence, the odd girth of $G(n,k)$ is $g = n/\gcd(n,k)$ and the inner cycle  $C = (v_0, v_k, \ldots, v_{kg} = v_0)$ is an odd cycle of length $g$. Consider now $h$ an endomorphism such that $h(v_0) = u_0$. Since $h$ is an endomorphism, $C' = (h(v_0) = u_0, h(v_k), \ldots, h(v_{kg}) = h(v_0) = u_0)$ has to be an odd cycle of length $g$ and thus, $C'$ has to be the outer cycle and $\gcd(n,k) = 1$. Moreover, by Theorem~\ref{thm:cores}.(c) we have that either $a$ is odd and $a \leq 1$, or $a$ is even and $a \geq n-1$; in both cases we get that $v_{n-1}v_0$ is an edge and, therefore, $k = 1$. As a consequence $G(n,k)$ is the $n$-prism and it is transitive. $\hfill \qed$

\section{Cayley graphs}
In this section we study the question which generalized Petersen graphs come from Cayley graphs of semigroups and monoids. Before describing our results, we introduce some definitions. The (right) Cayley graph $\Cay(S,C)$ of the semigroup $S$ with respect to the connection set $C\subseteq S$ is the directed looped multigraph with vertex set $S$ and one arc $(s,sc)$ for each $s\in S$ and $c\in C$. The \emph{underlying} graph of a directed looped multigraph is obtained by suppressing loops, forgetting orientations, and merging parallel edges into one.
We say that $G$ is a \emph{group graph}, \emph{monoid graph}, or \emph{semigroup graph}, if $G$ is the underlying graph of the Cayley graph of a group, monoid, and semigroup, respectively. If we want to specify a representation we say $G$ is a semigroup graph $\Cay(S,C)$, and similarly for the case of monoids and groups. 
In~\cite{Ned-95} generalized Petersen graph that are group graph are characterized (see also~\cite{Lov-97}).
\begin{theorem}\cite{Ned-95,Lov-97} \label{thm:cayleygroup} $G(n,k)$ is a group graph if and only if $k^2 \equiv 1 \ ({\rm mod}\ n)$. 
\end{theorem}

In the first part of this section we show that most generalized Petersen graphs that are cores, cannot be semigroup graphs unless $\Cay(S,C)$ has loops.
Together with Theorem~\ref{thm:cores} this gives an infinite such family and in particular by choosing any vertex transitive (multi)orientation yields vertex-transitive digraphs, that are not directed Cayley graphs of a semigroup. This answers a question of~\cite[Question 6.6]{garcamarco2019cayley} and strengthens a result of~\cite{Kho-21} for monoids.

In the second part of this section, we show for several generalized Petersen graphs that they are monoid graphs. Answering a question on \emph{mathoverflow} whether the Petersen graphs was Cayley graph of a group-like structure~\cite{390161}, we present four different ways to represent the Petersen graph as a Cayley graph (Proposition~\ref{prop:Petersen}). Furthermore, we show that the Kronecker cover of the Petersen graph -- the Desargues graph is the underlying graph of a monoid Cayley graph (Proposition~\ref{prop:Desargues}). Further, we answer a part of a question of~\cite[Question 4]{Kna-16} (see also~\cite[Section 13.1]{Kna-19}).  The graphs of all
Platonic solids are known to be group graphs with the sole exception of the dodecahedron, and it was asked whether it is a semigroup graph. In Proposition~\ref{pr:dodeca} we provide a positive answer, indeed, we prove that the dodecahedron $G(10,2)$ is a monoid graph ${\rm Cay}(M,C)$, where the connection set $C$ has $3$ elements and minimally generates $M$.
Finally, we show an infinite family generalized Petersen graphs, that are monoid graphs (Theorem~\ref{thm:Cay1}). Indeed, the used monoids are orthogroups, i.e. close to groups.
Apart from the above mentioned group graph characterization, the only results into this direction so far have been by Hao, Gao, and Luo~\cite{zbMATH05973394,zbMATH06093204} who show that every generalized Petersen graph appears as a certain subgraph of the Cayley graph of a symmetric inverse semigroup as well as a Brandt semigroup. However, these results have been improved recently by showing that every (directed) graph is a connected component of a monoid Cayley graph~\cite{Kna-21}.

In the last part of this section we show characterize generalized Petersen graphs that are monoid graphs with respect to a generating set of size $2$ (Theorem~\ref{thm:mon2gen}), and provide several properties and a conjecture about generalized Petersen graphs that are monoid graphs with respect to a connection set of size $2$.

\subsection{Cores and loopless semigroup graphs}

\begin{lemma} \label{lem:leftinCayley}
Let $D=\Cay(S,C)$ be a Cayley graph, then left multiplication by $S$ 
yields a homomorphism from $S$ to a subsemigroup of $\End(D)$.  
\end{lemma}
\begin{proof}
Let $\varphi: S \longrightarrow \End(D)$ defined as $\varphi(s): S \longrightarrow S$, where $\varphi(s)(s') = s \cdot s'$. 
Let $s \in S$, it is easy to see that $\varphi(s)$ is an endomorphism of $D$. Indeed, if $(u,v)$ is an arc of $D$, then  $v = u \cdot c$ for some $c \in C$. As a consequence $s \cdot v = s \cdot (u \cdot c) = (s \cdot u) \cdot c$ and, then, there is an arc from $s \cdot u$ to $s \cdot v$.
Moreover, we have that $\varphi(s \cdot s') = \varphi(s) \circ \varphi(s')$ because $S$ is associative.
\end{proof}

\begin{lemma}\label{lem:core}
If a core $G$ without four-cycles is a semigroup graph $\Cay(S,C)$, then $\Cay(S,C)$ has loops or $S$ is a group. 
\end{lemma}
\begin{proof}
 Suppose $D=\Cay(S,C)$ has no loops and $S$ is not a group. If $|C|\leq 1$, then $G$ either is a disjoint union of edges or pseudo-trees and not a core, or it is an odd cycle or an edge, in this case $S$ is the cyclic group of order $|V(D)|$ -- contradiction. Thus, $|C|\geq 2$. By Lemma~\ref{lem:leftinCayley} left multiplication by $S$ yields a homomorphism from $S$ to a subsemigroup of $\End(D)$. But Since $D$ has no loops, the latter equals $\Aut(D)$ since $G$ is a core. Thus, if $S$ is not a group, then two elements of $s,t\in S$ must have the same left-multiplication. Since left-multiplication is an automorphism of $D$, for any distinct $c,d\in C$, we have that $tc\neq td$. Thus, $s,sc=tc,t,td=sd$ is a four-cycle -- contradiction.
\end{proof}

\begin{corollary}\label{cor:petnotcayley}
If $G(n,k)$ is a core and $n \neq 4k$, then if $G(n,k)$ is a semigroup graph $Cay(S,C)$, then the latter has loops or $S$ is a group.
\end{corollary}
\begin{proof}
 It suffices to show that such graphs have no four-cycles. The rest follows from Lemma~\ref{lem:core}. So, assume that $G(n,k)$ has a four-cycle. If it involves inner and outer vertices then it is an $n$-prism and, thus, not a core. If it only involves inner or outer vertices, then we have that $4 = n / \gcd(n,k)$ and, since $0 < k < n/2$, then $n = 4k$. 
\end{proof}

Corollary~\ref{cor:petnotcayley} provides us with an infinite family of negative instances of~\cite[Question 6.6]{garcamarco2019cayley}:
\begin{corollary}\label{cor:manybad}
 There are infinitely many vertex-transitive digraphs, that are not the Cayley digraph of a semigroup. 
\end{corollary}
\begin{proof} Take any graph such that: it is transitive, it is a core, has no $4$-cycles, and is not a group graph. Now, consider a biorientation of it, i.e., replace each edge by two oppositely oriented arcs. This digraph is vertex transitive and, by Lemma~\ref{lem:core}, it is not the directed Cayley graph of a semigroup. Let us see that there are infinite graphs satisfying these conditions within the family of generalized Petersen graphs.

For $(n,k) = (10,2)$ we have that $G(10,2)$ is transitive (Corollary~\ref{cor:autdi}), it is a core (Theorem~\ref{thm:cores}), has no $4$-cycles and is not a group graph (Theorem~\ref{thm:cayleygroup}).

If $n$ is odd and $k^2 \equiv -1\ ({\rm mod}\ n)$, we have that $G(n,k)$ is transitive (Corollary~\ref{cor:autdi}), it is not bipartite (because $n$ is odd) and $\gcd(n,k) = 1$ , then it is a core (Corollary~\ref{cor:primos}), if has no $4$-cycles (because $n \neq 4k$) and is not a group graph (Theorem~\ref{thm:cayleygroup}). This is an infinite family.
\end{proof}

Clearly, the digraphs above have arcs in both directions and one could wonder whether this is essential for such a construction. A vertex-transitive digraph has at each vertex the same outdegree which also equals the indegree. Thus, if we want to have an example without multiple arcs, its underlying undirected graph has to be regular of even degree and thus cannot be found among generalized Petersen graphs. We believe, however that such graphs should be easy to find as well.

\subsection{Positive results}
In this section we study generalized Petersen graphs that are underlying graphs of Cayley graphs of semigroups or monoids. 

Let us start with four semigroup representations of the Petersen graph $G(5,2)$. The semigroups $S,M,S',M'$ are given in Table~\ref{tab:Pet}. They yield the Petersen graph as their Cayley graph as depicted in Figure~\ref{fig:petersen}, where also the connection sets are specified. Both $S,M$ are unions of $\mathbb{Z}_6$ and the \emph{null semigroup} $N_{[6,9]}$, i.e., $ab=9$ for all $a,b\in \{6,\ldots, 9\}$. Moreover, $M$ is a monoid with neutral element $0$. 
Similarly, $S',M'$ are unions of the dihedral group $D_3$ of order $6$ and the \emph{null semigroup} $N_{[6,9]}$ and $M'$ is a monoid with neutral element $5$.

\begin{table}
\parbox{.45\linewidth}{
\centering
\begin{tabular}{l|cccccccccc } 
 
 $S$&0&1&2&3&4&5&6&7&8&9 \\ 
 \hline
 0&0&1&2&3&4&5&6&7&8&9 \\ 
 1&1&2&3&4&5&0&7&8&6&9 \\ 
 2&2&3&4&5&0&1&8&6&7&9 \\ 
 3&3&4&5&0&1&2&6&7&8&9 \\ 
 4&4&5&0&1&2&3&7&8&6&9 \\ 
 5&5&0&1&2&3&4&8&6&7&9 \\ 
 6&9&9&9&9&9&9&9&9&9&9 \\ 
 7&9&9&9&9&9&9&9&9&9&9 \\ 
 8&9&9&9&9&9&9&9&9&9&9 \\ 
 9&9&9&9&9&9&9&9&9&9&9 \\ 
\end{tabular}\label{tab:Pet1}
}
\hfill
\parbox{.45\linewidth}{
\centering

\begin{tabular}{l|cccccccccc } 
 
 $M$&0&1&2&3&4&5&6&7&8&9 \\ 
 \hline
 0&0&1&2&3&4&5&6&7&8&9 \\ 
 1&1&2&3&4&5&0&7&8&6&9 \\ 
 2&2&3&4&5&0&1&8&6&7&9 \\ 
 3&3&4&5&0&1&2&6&7&8&9 \\ 
 4&4&5&0&1&2&3&7&8&6&9 \\ 
 5&5&0&1&2&3&4&8&6&7&9 \\ 
 6&6&6&6&6&6&6&9&9&9&9 \\ 
 7&7&7&7&7&7&7&9&9&9&9 \\ 
 8&8&8&8&8&8&8&9&9&9&9 \\ 
 9&9&9&9&9&9&9&9&9&9&9 \\ 
\end{tabular}\label{tab:Pet2}
}
\end{table}

\begin{table}
\parbox{.45\linewidth}{
\centering
\begin{tabular}{l|cccccccccc } 
 
 $S'$&0&1&2&3&4&5&6&7&8&9 \\ 
 \hline
 0&5&4&3&2&1&0&8&7&6&9 \\ 
 1&2&3&4&5&0&1&8&6&7&9 \\ 
 2&1&0&5&4&3&2&7&6&8&9 \\ 
 3&4&5&0&1&2&3&7&8&6&9 \\ 
 4&3&2&1&0&5&4&6&8&7&9 \\ 
 5&0&1&2&3&4&5&6&7&8&9 \\ 
 6&9&9&9&9&9&9&9&9&9&9 \\ 
 7&9&9&9&9&9&9&9&9&9&9 \\ 
 8&9&9&9&9&9&9&9&9&9&9 \\ 
 9&9&9&9&9&9&9&9&9&9&9 \\ 
\end{tabular}\label{tab:Pet4}

}
\hfill
\parbox{.45\linewidth}{
\centering
\begin{tabular}{l|cccccccccc } 
 
 $M'$&0&1&2&3&4&5&6&7&8&9 \\ 
 \hline
 0&5&4&3&2&1&0&8&7&6&9 \\ 
 1&2&3&4&5&0&1&8&6&7&9 \\ 
 2&1&0&5&4&3&2&7&6&8&9 \\ 
 3&4&5&0&1&2&3&7&8&6&9 \\ 
 4&3&2&1&0&5&4&6&8&7&9 \\ 
 5&0&1&2&3&4&5&6&7&8&9 \\ 
 6&6&6&6&6&6&6&9&9&9&9 \\ 
 7&7&7&7&7&7&7&9&9&9&9 \\ 
 8&8&8&8&8&8&8&9&9&9&9 \\ 
 9&9&9&9&9&9&9&9&9&9&9 \\ 
\end{tabular}\label{tab:Pet3}
}
\caption{Different ways to realize the Petersen graph}
\label{tab:Pet}
\end{table}


\begin{figure}[ht]
\begin{center}
\includegraphics[width =.7\textwidth]{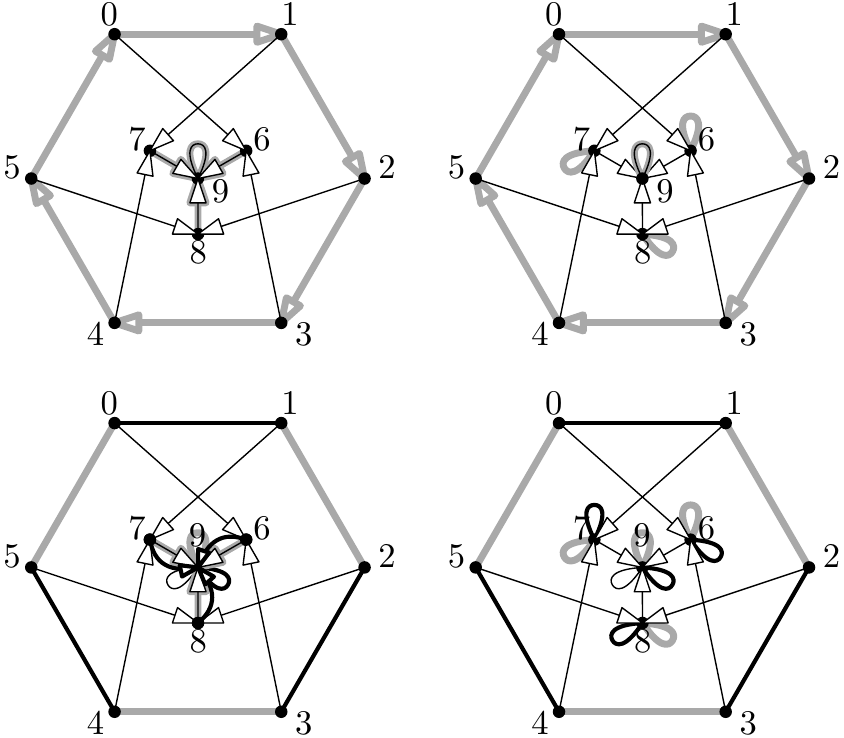}
\caption{From upper left to bottom right: $\Cay(S,\{1,6\})$, $\Cay(M,\{1,6\})$, $\Cay(S',\{0,4,8\})$, $\Cay(M',\{0,4,8\})$.}
\label{fig:petersen}
\end{center}
\end{figure}

Together with Table~\ref{tab:Pet} and Figure~\ref{fig:petersen} we conclude the above discussion:

\begin{proposition}\label{prop:Petersen}
 The Petersen graph $G(5,2)$ is a monoid graph.
\end{proposition}
%

The following is straight-forward and will be useful to show that the Desargues graph is a monoid graph:
\begin{lemma}\label{lem:monoid}
 Let $R=S\dot\cup T$ be a semigroup such that $ST\subseteq T$ and $TS\subseteq T$ and $R'$ another semigroup. The set $R\times R'$ is a semigroup via $(s,i)(r,j)=(sr,i)$ and $(t,i)(r,j)=(tr,ij)$, for all $r\in R, s\in S, t\in T, i,j\in R'$ and the natural multiplication within $R$ and $R'$, respectively. If both $T$ and $R'$ are monoids, then so is the resulting semigroup.
\end{lemma}
\begin{proof}
 We check associativity, where clearly the case where all three elements come from $S\times R'$ or $T\times R'$ , respectively, can be ignored because on these sets we have semigroup structure by hypothesis. The other six cases are straight-forward computations:\small{
 $$((s,i)(t,j))(t',k)=(st,i)(t',k)=(stt',i)=(s,i)(tt',jk)=(s,i)((t,j)(t',k)),$$
 $$((t,i)(s,j))(t',k)=(ts,ij)(t',k)=(tst',ij)=(t,i)(st',j)=(t,i)((s,j)(t',k)),$$
 $$((t,i)(t',j))(s,k)=(tt',ij)(s,k)=(tt's,ijk)=(t,i)(t's,jk)=(t,i)((t',j)(s,k)),$$
 $$((t,i)(s,j))(s',k)=(ts,ij)(s',k)=(tss',ij)=(t,i)(ss',j)=(t,i)((s,j)(s',k)),$$
 $$((s,i)(t,j))(s',k)=(st,i)(s',k)=(sts',i)=(s,i)(ts',jk)=(s,i)((t,j)(s',k)),$$
 $$((s,i)(s',j))(t,k)=(ss',i)(t,k)=(ss't,i)=(s,i)(s't,j)=(s,i)((s',j)(t,k)).$$}
 
 Finally, if $e\in T$ and $e'\in R'$ are neutral elements, we clearly have $(e,e')(r,j)=(er,e'j)$, since $e\in T$. furthermore $(r,j)(e,e')=(r,j)$ independently of whether $r\in S$ or $r\in T$.
\end{proof}

\begin{proposition}\label{prop:Desargues}
 The Desargues graph $G(10,3)$ is a monoid graph.
\end{proposition}
\begin{proof}
 The Cayley graph is depicted in the left of Figure~\ref{fig:desargues}. Let us see that this really is the Cayley graph of a monoid. In fact consider the monoid representation $\Cay(M,\{1,6\})$ of the Petersen graph, where $M=\mathbb{Z}_6\cup N_{[6,9]}$. Note in particular, that we have $\mathbb{Z}_6\cdot N_{[6,9]}\subseteq N_{[6,9]}$ and $N_{[6,9]}\cdot\mathbb{Z}_6\subseteq N_{[6,9]}$. Hence by Lemma~\ref{lem:monoid} the set $M\times\mathbb{Z}_2$ carries a monoid structure $\widetilde{M}$. The graph in the left of Figure~\ref{fig:desargues} is $\Cay(\widetilde{M},\{(1,1),(6,0)\})$. Note however that $\{(1,1),(6,0)\}$ does not generate $\widetilde{M}$. 
\end{proof}

\begin{figure}[ht]
\begin{center}
\includegraphics[width =\textwidth]{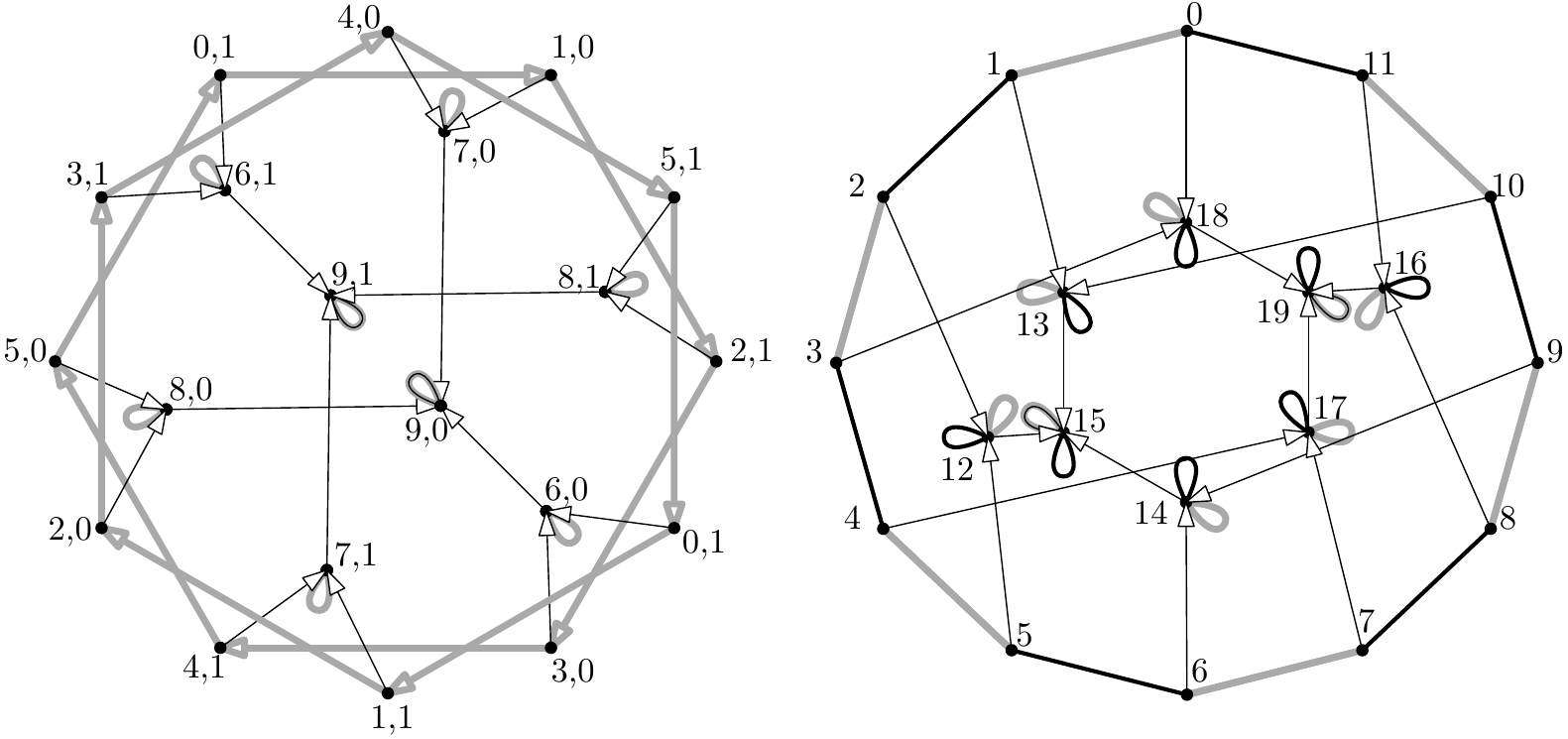}
\caption{Left: A Cayley graph realizing the Desargues graph $G(10,3)$. Right: A Cayley graph realizing the Dodecahedron graph $G(10,2)$. }
\label{fig:desargues}
\end{center}
\end{figure}

Now, let us consider the monoid $M$ depicted in Table~\ref{tab:dodecahedron}. This monoid is the disjoint union of the dihedral group $D_6$ on $\{0,\ldots 11\}$, the and the two null semigroups $N_{[12,15]}$ and $N_{[16,19]}$. The Cayley graph $\Cay(M,\{1,11,18\})$ depicted on the right of Figure~\ref{fig:desargues} realizes the Dodecahedron graph $G(10,2)$. We get:

\begin{proposition}\label{pr:dodeca}
 The Dodecahedron graph $G(10,2)$ is a monoid graph.
\end{proposition}

\begin{table}

\centering
\begin{tabular}{l|cccccccccccccccccccc }

$M$&0&1&2&3&4&5&6&7&8&9&10&11&12&13&14&15&16&17&18&19\\ 
\hline 
0&0&1&2&3&4&5&6&7&8&9&10&11&12&13&14&15&16&17&18&19\\ 
1&1&0&11&10&9&8&7&6&5&4&3&2&16&18&17&19&12&14&13&15\\ 
2&2&3&4&5&6&7&8&9&10&11&0&1&17&18&16&19&13&14&12&15\\ 
3&3&2&1&0&11&10&9&8&7&6&5&4&13&12&14&15&17&16&18&19\\
4&4&5&6&7&8&9&10&11&0&1&2&3&14&12&13&15&18&16&17&19\\ 
5&5&4&3&2&1&0&11&10&9&8&7&6&18&17&16&19&14&13&12&15\\ 
6&6&7&8&9&10&11&0&1&2&3&4&5&16&17&18&19&12&13&14&15\\ 
7&7&6&5&4&3&2&1&0&11&10&9&8&12&14&13&15&16&18&17&19\\ 
8&8&9&10&11&0&1&2&3&4&5&6&7&13&14&12&15&17&18&16&19\\ 
9&9&8&7&6&5&4&3&2&1&0&11&10&17&16&18&19&13&12&14&15\\ 
10&10&11&0&1&2&3&4&5&6&7&8&9&18&16&17&19&14&12&13&15\\ 
11&11&10&9&8&7&6&5&4&3&2&1&0&14&13&12&15&18&17&16&19\\ 
12&12&12&12&12&12&12&12&12&12&12&12&12&15&15&15&15&15&15&15&15\\ 
13&13&13&13&13&13&13&13&13&13&13&13&13&15&15&15&15&15&15&15&15\\
14&14&14&14&14&14&14&14&14&14&14&14&14&15&15&15&15&15&15&15&15\\ 
15&15&15&15&15&15&15&15&15&15&15&15&15&15&15&15&15&15&15&15&15\\ 
16&16&16&16&16&16&16&16&16&16&16&16&16&19&19&19&19&19&19&19&19\\ 
17&17&17&17&17&17&17&17&17&17&17&17&17&19&19&19&19&19&19&19&19\\ 
18&18&18&18&18&18&18&18&18&18&18&18&18&19&19&19&19&19&19&19&19\\ 
19&19&19&19&19&19&19&19&19&19&19&19&19&19&19&19&19&19&19&19&19\\
\end{tabular}
\caption{A monoid representing of the Dodecahedron.}\label{tab:dodecahedron}
\end{table}

%

\bigskip

After having examined three particular generalized Petersen graphs, we proceed to construct an infinite family  of generalized Petersen graphs that are monoid graphs.
In the following we show that if $k^2 = \pm k \mod n$, then $G(n,k)$ is a loopless monoid graph. For instance the D\"urer graph $G(6,2)$ falls into this family and another example is displayed in Figure~\ref{fig:G124}.
Before stating the result, we need one more lemma, that might be of independent use. Recall that the the \emph{left-zero-band} $L_I$ is defined on $\{\ell_i\mid i\in I\}$ via $\ell_i\ell_j=\ell_i$ for all $i,j\in I$.

\begin{lemma}\label{lem:homomomo}
 If $S,T,R$ are semigroups and $\varphi:S\to T$ and $\psi:S\to R$ semigroup homomorphism. Then $S\cup (T\times L_{R})$ carries a semigroup structure via $s(t,\ell_{r})=(\varphi(s)t,\ell_{\psi(s)r})$ and $(t,\ell_{r})s=(t\varphi(s),\ell_{r})$ and the natural multiplication within $S$ and $T\times L_{T}$, respectively.
\end{lemma}
\begin{proof}
We check associativity, where clearly the case where all three elements come from $S$ or $T\times L_{R}$, respectively, can be ignored because on these sets we have semigroup structure by hypothesis. The other six cases are straight-forward computations:\small{
$$(s(t,\ell_{r}))(t',\ell_{r'})=(\varphi(s)t,\ell_{\psi(s)r})(t',\ell_{r'})=(\varphi(s)tt',\ell_{\psi(s)r})=s(tt',\ell_{r})=s((t,\ell_{r})(t',\ell_{r'})),$$
$$((t,\ell_{r})s)(t',\ell_{r'})=(t\varphi(s),\ell_{r})(t',\ell_{r'})=(t\varphi(s)t',\ell_{r})=(t,\ell_{r})(\varphi(s)t',\ell_{\psi(s)r'})=(t,\ell_{r})(s(t',\ell_{r'})),$$
$$((t,\ell_{r})(t',\ell_{r'}))s=(tt',\ell_{r})s=(tt'\varphi(s),\ell_{r})=(t,\ell_{r})(t'\varphi(s),\ell_{r'})=(t,\ell_{r})((t',\ell_{r'})s),$$
$$((t,\ell_{r})s)s'=(t\varphi(s),\ell_{r})s'=(t\varphi(s)\varphi(s'),\ell_{r})=(t\varphi(ss'),\ell_{r})=(t,\ell_{r})(ss'),$$
$$(s(t,\ell_{r}))s'=(\varphi(s)t,\ell_{\psi(s)r})s'=(\varphi(s)t\varphi(s'),\ell_{\psi(s)r})=s(t\varphi(s'),\ell_{r})=s((t,\ell_{r})s'),$$
$$(ss')(t,\ell_{r})=(\varphi(ss')t,\ell_{\psi(ss')r})=(\varphi(s)\varphi(s')t,\ell_{\psi(s)\psi(s')r})=s(\varphi(s')t,\ell_{\psi(s')r})=s(s'(t,\ell_{r})).$$}
\end{proof}

Recall, that a semigroup $S$ is an \emph{orthogroup} if $S$ is the union of groups and its idempotent elements form a subsemigroup, see e.g.~\cite{Kil-00}. Note that none of the semigroups we have seen in this section so far is an orthogroup.

\begin{theorem}\label{thm:Cay1}
If $k^2 \equiv \pm k \mod n$, then $G(n,k)$ is a monoid graph $\Cay(M,C)$ where the latter is loopless and $M$ is an orthogroup.
\end{theorem}
\begin{proof}First observe that $k^2 = \pm k \mod n$ is equivalent to $k \equiv \pm 1 \mod n/\gcd(n,k)$. Consider $S=A\cup A'$ with $A=\mathbb{Z}_{n}$ and $A'=\mathbb{Z}_{\frac{n}{\gcd(n,k)}}\times L_{\gcd(n,k)}$. Here $L_{\gcd(n,k)}:=L_{\mathbb{Z}_{\gcd(n,k)}}$. 
Since $A$ is a group and $A'$ a \emph{left-group}, i.e., the product of a group and a left-zero-band, we have that $S$ is the union of groups. This already yields one of the properties required for an orthogroup.

Now, for $x\in\mathbb{Z}_{n}$ and $(i,\ell_j)\in \mathbb{Z}_{\frac{n}{\gcd(n,k)}}\times L_{\gcd(n,k)}$ define $$x(i,\ell_j)=(x+i \mod \frac{n}{\gcd(n,k)},\ell_{x+j \mod \gcd(n,k)})$$ and $$(i,\ell_j)x=(x+i \mod \frac{n}{\gcd(n,k)},\ell_j).$$

Note that defining $\varphi:\mathbb{Z}_{n}\to \mathbb{Z}_{\frac{n}{\gcd(n,k)}}$ as $x\mapsto x \mod \frac{n}{\gcd(n,k)}$ and $\psi:\mathbb{Z}_{n}\to \mathbb{Z}_{{\gcd(n,k)}}$ as $x\mapsto x \mod {\gcd(n,k)}$ we get two semigroup homomorphisms. By Lemma~\ref{lem:homomomo} our operation is a semigroup. Further note that $0\in\mathbb{Z}_n$ is a neutral element of this operation, so we do have a monoid. Finally, the set of idempotent elements $I(S)$ of $S$ consists  of $0\in\mathbb{Z}_n$ and furthermore the set $\{0\}\times L_{\gcd(n,k)}$. Clearly, $I(S)\cong L^+_{\gcd(n,k)}$, i.e., $L_{\gcd(n,k)}$ with an adjoint neutral element. In particular $I(S)<S$ is a subsemigroup, which concludes the proof that $S$ is an orthogroup monoid.

Let now $C=\{1,(1,\ell_0)\}\subset S$. Clearly, $\Cay(S, C)$ is loopless. Let us see that $G(n,k)$ is the underlying graph of $\Cay(S, C)$. We identify $A$ with the outer vertices and $A'$ with the inner vertices. Clearly $1$ generates the outer-rim on $A$ and for each of the vertices $x \in A$  there is exactly one edge towards the inner vertices generated by $(1,\ell_0)$ and connecting $x$ with $(x+1,\ell_x)$ . Moreover both $1$ and $(1,\ell_0)$ have the same right action on $A'$ and partition the inner vertices into the $\gcd(n,k)$ cycles of length $\frac{n}{\gcd(n,k)}$.
We also observe that for all $i \in \{0,\ldots, \gcd(n,k)-1\}$, the inner neighbors of the vertices $\{x \, \vert \, x \equiv i \ ({\rm mod}\ \gcd(n,k))\} \subseteq A$ are the vertices of one of the inner cycles, more precisely \[ \{x(1,\ell_0) \, \vert \, x \equiv i \ ({\rm mod}\ \gcd(n,k))\} = \mathbb Z_{n/\gcd(n,k)} \times \{\ell_i\} \subseteq A'.\] Finally, we have that $x (1,\ell_0)$ (the inner neighbour of $x \in A$) and $(x+k) (1,\ell_0)$ (the inner neighbor of $x+k \in A$) are also neighbors, indeed, \begin{itemize} \item if $k \equiv 1 ({\rm mod}\ (n/\gcd(n,k)))$, then \[ (x + k)(1,\ell_0) = (x+k+1, \ell_{x+k}) = (x+2,\ell_x) = x (1,\ell_0) (1, \ell_0),\] \item if $k \equiv -1 ({\rm mod}\ (n/\gcd(n,k)))$, then  \[ (x + k)(1,\ell_0)(1,\ell_0) = (x+k+2,\ell_{x-k}) = (x+1,\ell_x) =  x (1,\ell_0).\] \end{itemize}
This completes the proof.  
See the left of Figure~\ref{fig:G124} for an example. 
\end{proof}

\begin{remark}\label{rem:variants}
In the above construction we get interior double arcs which are parallel when $k^2 \equiv k \ ({\rm mod} \ n)$,  and anti-parallel when $k^2 \equiv -k \ ({\rm mod} \ n)$. If we chose $C=\{1,(-1,\ell_0)\}$ in that construction the situation gets reversed, that is, we get the same underlying graph but the interior double arcs are anti-parallel  for $k^2 \equiv k \ ({\rm mod} \ n)$ and parallel when $k^2 \equiv -k \ ({\rm mod} \ n)$. Moreover, if we choose $C=\{1,(0,\ell_0)\}$ we obtain a digraph with loops on the inner vertices but without multiple arcs whose underlying graph is $G(n,k)$. See Figure~\ref{fig:G124}. 
\end{remark}

\begin{figure}[ht]
\begin{center}
\includegraphics[width = \textwidth]{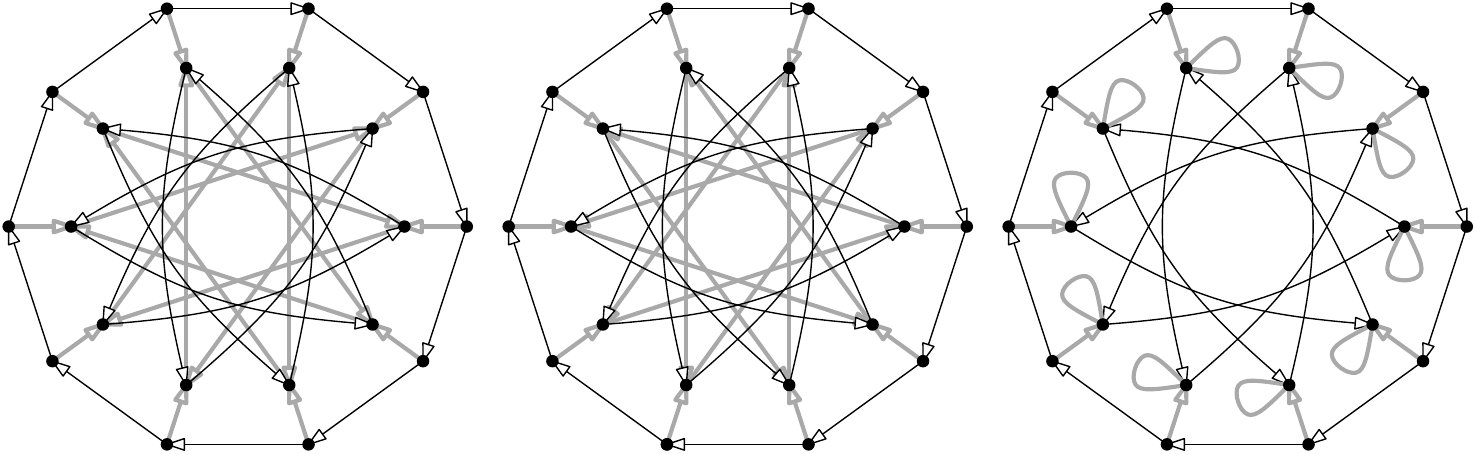}
\caption{Three Cayley graphs of the monoid from Theorem~\ref{thm:Cay1} with underlying graph $G(10,4)$, also see Remark~\ref{rem:variants}.}
\label{fig:G124}
\end{center}
\end{figure}
 
%
%
%
%
%

\subsection{Generalized Petersen graphs that are $2$-generated monoid graphs}

The degree of the vertex $e$ of a monoid graph ${\rm Cay}(M,C)$ is at least $|C|$ and the number of edges of the graph is at most $|C|$ times its number of vertices. Hence, if a cubic graph is a monoid graph, then $2\leq |C|\leq 3$. 
The goal of this section is to study the case $|C|=2$. We prove Theorem~\ref{thm:mon2gen}, which characterizes all the generalized Petersen graphs that are monoid  graphs with underlying Cayley graph ${\rm Cay}(M,C)$ with $M = \langle C \rangle$ and $|C| = 2$. These graphs are exactly the ones obtained in the Theorem~\ref{thm:Cay1}, the group Cayley graphs, and the Petersen graph. 

\begin{theorem}\label{thm:mon2gen}The generalized Petersen graph $G(n,k)$ is a monoid graph ${\rm Cay}(M,C)$ with $M = \langle C \rangle$ and $|C| = 2$ if and only if one of the following holds: \begin{itemize} \item[(a)] $(n,k)=(5,2)$ (Petersen graph),  \item[(b)] $k^2 \equiv 1 \ ({\rm mod}\ n),$ or \item[(c)] $k^2 \equiv \pm k\ ({\rm mod}\ n)$. \end{itemize}
\end{theorem}

\begin{lemma}\label{lm:ciclo2} Let $G = \Cay(M,C)$ be a cubic graph with $|C| = 2$, then there exists an invertible element $g \in C$ of order $o(g) > 2$. Moreover, $(e,g,g^2,\ldots,g^{o(g)-1},e)$ is a cycle in $G$. 
\end{lemma}
\begin{proof} We observe that $C \subseteq N_G(e)$. Moreover, since $G$ is a cubic graph and $|C| = 2$, then there exist $x \notin C$ and $g \in C$ such that $x g = e$. Let us see that $g$ is invertible. We take $i \in \mathbb N$ the minimum value such that there exists a $j > i$ such that $g^i = g^j$ and let us see that $i = 0$. Indeed, if $i > 0$, then $g^{i-1} = (x g) g^{i-1} = x g^i = x g^j = (x g)g^{j-1} = g^{j-1}$, a contradiction. As a consequence $g$ is invertible. Finally, we have that $g^2 \neq e$ because, otherwise $x = x g^2 = (xg) g = g \in C$, a contradiction.
Clearly $(e,g,g^2,\ldots,g^{o(g)-1},e)$ is a cycle in $G$.
\end{proof}

In this section we will again make use of the rotation $\alpha$, reflection $\beta$, and the inside-out map $\gamma$ from the automorphism group of the generalized Petersen graph.
In~\cite{Ned-95} (see also~\cite{Lov-97}), it is proven that $G(n,k)$ is a group graph if and only if  $k^2 \equiv 1 \ ({\rm mod}\ n)$. In~\cite{Lov-97}, the author observes that whenever $G(n,k)$ is a group graph, it is $\Cay(H,C)$ with $H = \langle \alpha, \gamma  \, \vert\, \alpha^n = \gamma^2 = {\rm id},\,  \gamma \alpha \gamma = \alpha^k \rangle$ and $C = \{\alpha,\gamma\}$.  Hence, one gets the following.

\begin{corollary}\label{cor:petgroup} If $k^2 \equiv 1 \ ({\rm mod}\ n)$, then $G(n,k) = \Cay(H,C)$ with $H = \langle C \rangle$ a group and $|C| = 2.$ 
\end{corollary}

For a Cayley graph ${\rm Cay}(M,C)$, a color endomorphism is a graph endomorphism $\phi: M~\longrightarrow~M$ such that $\phi(m) c = \phi(m c)$ for all $m \in M$, $c \in C$. Another ingredient we will use in the proof of Theorem~\ref{thm:mon2gen} is the following variant of Lemma~\ref{lem:homomomo}.
\begin{theorem}\cite[Theorem 7.3.7]{Kna-19}\label{colorend} Let $M$ be a monoid with generating set $C \subseteq M$. Then, $M$ is isomorphic to the monoid of color endomorphism of ${\rm Cay}(M,C)$. Moreover, the isomorphism is given by $m \mapsto \lambda_m$, being  $\lambda_m$ the left-multiplication, i.e., $\lambda_m: M \longrightarrow M$ with $\lambda_m(m') = mm'$.
\end{theorem}

The statement of the following lemma is similar to the one of Theorem~\ref{thm:mon2gen} but removing the hypothesis that the connection set $C$ generates $M$. It will be used in all the main results of this section.

\begin{lemma}\label{lm:technical} If $G(n,k)$ is a monoid graph ${\rm Cay}(M,C)$ with $|C| = 2$, then one of the following holds: \begin{itemize} \item[(a)] $(n,k)=(5,2)$ (Petersen graph),  \item[(b)] $(n,k)=(10,3)$ (Desargues graph) and there is an invertible $g \in C$ of order $6$, \item[(c)] $k^2 \equiv 1 \ ({\rm mod}\ n),$ \item[(d)] $k^2 \equiv \pm k\ ({\rm mod}\ n),$ or \item[(e)] $\gcd(n,k) \neq 1$ and there exists $g \in C$ such that $\lambda_g = \alpha^k$ or $\lambda_g = \alpha^{-k}$, where $\lambda_g$ and $\alpha$ are left-multiplication and rotation, respectively. \end{itemize}
\end{lemma}
\begin{proof}Let $G(n,k)$ be a monoid graph that is the underlying graph of ${\rm Cay}(M,C)$, where $M$ is a monoid and $|C| = 2$, and assume that $G(n,k)$ is not a group Cayley graph and $(n,k) \not= (5,2)$.  By  Lemma~\ref{lm:ciclo2}, there is an invertible element $g \in C$ such that $g^2 \neq 1$. In particular,
\begin{itemize}
\item[(1)] there exists an induced cycle $C' = (x_0,\ldots,x_{\ell-1},x_{\ell} = x_0)$ of length $\ell := o(g)$
\item[(2)] there exists $\tau \in {\rm Aut}(G(n,k))$ of order $\ell$ such that $\tau(x_i) = x_{i+1}$ for $i = 0,\ldots,\ell-1$.
\end{itemize}
Indeed, $C' = (e,g,g^2,\ldots,g^{o(g)-1},e)$ satisfies (1), and $\lambda_g$ satisfies (2).   

\noindent \emph{Claim: }Either (b) holds or there exists $\varphi \in {\rm Aut}(G(n,k))$ such that $\varphi(C')$ is either the exterior or an interior cycle.

\noindent \emph{Proof of the claim.}    We separate four cases:

\emph{Case $(n,k) = (10,2)$.} A computer assisted exhaustive search using SageMath \cite{sage} through all the automorphisms of $G(10,2)$ (its automorphism group is isomorphic to $A_5 \times \mathbb Z_2$) shows that the only values $\ell$ such that (1) and (2) hold are $\ell = 5$ and $\ell = 10$. Moreover,  for $\ell = 5$,  there is an automorphism $\varphi$ such that $\varphi(C')$ is an inner cycle, and  for $\ell = 10$,  there is an automorphism $\varphi$ such that $\varphi(C')$ is the outer cycle.

\emph{Case $(n,k) = (10,3)$.} Again using SageMath \cite{sage} we analyze the automorphisms of $G(10,3)$ (its automorphism group is isomorphic to $S_5 \times \mathbb Z_2$) to find that the only values $\ell$ such that (1) and (2) hold are $\ell = 6$ and $\ell = 10$. If $\ell = 6$, then (b) holds. Moreover, if $\ell = 10$, then there is an automorphism $\varphi$ such that $\varphi(C')$ is the outer cycle.

	\emph{Case $k^2 \not\equiv \pm 1 \ ({\rm mod}\ n)$.} By  Corollary~\ref{cor:autdi} we have that ${\rm Aut}(G(n,k)) = \langle \alpha, \beta \rangle \cong D_{n}$, and the only elements of order $> 2$ are of the form $\alpha^i$, then $\lambda_g = \alpha^i$. Since $C'$ is a cycle containing the edge $\{e,g\} = \{e, \alpha^i(e)\}$, then either $i \in \{1,n-1\}$ and $C_g$ is the exterior cycle, or $i \in \{k,n-k\}$ and $C_g$ is an interior cycle.

\emph{Case $k^2 \equiv -1 \ ({\rm mod}\ n)$ and $(n,k) \neq (10,3)$.} Since $\gcd(n,k) = 1$ and $k \neq 1$, then there are no $3$-cycles or $4$-cycles in $G(n,k)$. Hence $C'$ is a cycle of length $o(g) > 4$. By Theorem~\ref{thm:autgroup}, every $\sigma \in {\rm Aut}(G(n,k))$, can be written as $\sigma =  \alpha^i \gamma^j$ for some $i \in \{0,\ldots,n-1\}$ and $j \in \{0,1,2,3\}$. Moreover, if $j \neq 0$, then $k^{3j} + k^{2j} + k^j + 1 \equiv (-1)^j k^{j} + k^j + (-1)^j + 1 \equiv 0 \ ({\rm mod}\ n)$ and $\sigma^4 = {\rm id}$. Hence, the only elements of order $> 4$ in ${\rm Aut}(G(n,k))$ are of the form $\alpha^i$ for some $1 \leq i < n$. Hence we get that $\lambda_g = \alpha^i$. Proceeding as in the previous case we get that $C_g$ is  either the exterior or the interior cycle. 

Thus, the claim follows.

%
%

As a consequence, one can assume without loss of generality that $C'$ itself is either the exterior cycle (and $\lambda_g = \alpha$ or $\lambda_g = \alpha^{-1}$), or an interior cycle (and $\lambda_g = \alpha^k$ or $\lambda_g = \alpha^{-k}$). 

If $C'$ has length $< n$, then (e) holds because $C'$ is necessarily an interior cycle and $\gcd(n,k) > 1$. Hence, it remains to consider when $C'$ is a a cycle of length $n$. We may also assume  without loss of generality that it is the external cycle (if $\gcd(n,k) = 1$, taking $k'$ the inverse of $k$ modulo $n$ we have that $G(n,k) \cong G(n,k')$ and the isomorphism interchanges inner and outer vertices) and $\lambda_g = \alpha$ or $\lambda_g= \alpha^{- 1}$.

We denote $C = \{g,h\}$. We observe that $h$ is not invertible; otherwise $M = \langle g,h \rangle$ is a group.  In particular, $h \notin \{g,g^{-1}\}$ and, then, $h$ is the inner neighbor of $e$. 
 Moreover,  we have that $g^i h \notin V_O$ for all $i$, so $g^i h$ is the inner neighbor of $g^i$. As a consequence the vertex set of the inner cycle containing $h$ is $V' = \{g^{\lambda k} h \, \vert \, 0 \leq \lambda < n/ \gcd(n,k)\}$.  We split the proof in two cases.

\emph{Case 1:} $h g = h$.  As a consequence, for all $x \in V'$ we have $x g = x$.  Consider the graph $D = {\rm Cay}(V', \{h\})$, we know that $D$ has to be a directed cycle (all the vertices in $D$ have out-degree $1$ and the underlying undirected graph is a cycle) and, hence, $V' = \{h^i \, \vert \, i \in \mathbb N\}$ is a set with at least $3$ elements. Depending on the orientation of $D$, either $h^2 = g^k h$ or $h^2 = g^{-k}h$. In the first case, using that $h g^k = h$ we get  $h^3 = h (g^k h) = (h g^k) h = h^2$, a contradiction.  In the second one we have that $h^2 = g^{-k} h$ and using again that $h g^k = h$ we get  $h^3 = (h g^k) h^2 = h (g^k h^2) = h^2$, a contradiction.

\emph{Case 2:} $hg \neq h$. Since $hg \neq e$, then $hg$ is an inner neighbor of $h$, we have that $hg = g^{\pm k} h$.
We also take $\mu \in \{-1,0,1\}$ so that $h^2 = h g^\mu$.  We claim that $h y  = y h$ for all $y \in V'$. Indeed, if $y = h g^i = g^{\pm k i} h \in V',$ then $h y = h (hg^i) = h^2 g^i = h g^{i+ \mu} = y g^{\mu}$ and $y h = g^{\pm ki } h^2 = g^{\pm ki} (h g^\mu) = (g^{\pm ki} h) g^{\mu} =  y g^{\mu}$. Taking $y = hg  \in V'$, we have that
\begin{itemize}
\item $ y h = (hg) h =  g^{\pm k} h^2 = g^{\pm k} h g^{\mu} = g^{\pm k} g^{\pm k \mu}h,  $ and
\item $ h y = h (hg) = h (g^{\pm k} h) = g^{(\pm k)^2} h^2 = g^{k^2} g^{\pm k \mu} h$.
\end{itemize}
 Using that $y h = hy$ we finally get that $g^{\pm k} = g^{k^2}$. Hence $k^2 \equiv \pm k \ ({\rm mod}\ n)$.
 
\end{proof}

Now we can proceed with the proof of the main result.

\noindent {\it Proof of Theorem~\ref{thm:mon2gen}} $(\Longleftarrow)$ Follows from Proposition~\ref{prop:Petersen} (for $G(5,2)$), from Corollary~\ref{cor:petgroup} (for $k^2 \equiv 1 \ ({\rm mod}\ n)$), and from Theorem~\ref{thm:Cay1}  (for $k^2 \equiv \pm k \ ({\rm mod}\ n)$).

\noindent  $(\Longrightarrow)$ It only remains to prove that (b) and (e) in Lemma~\ref{lm:technical} cannot hold.

Assume (b) holds and that $\tau\in\Aut(G(10,3))$ satisfies conditions (1) and (2) in the proof of Lemma~\ref{lm:technical} with respect to a $6$-cycle. An exhaustive search with SageMath \cite{sage} shows, that up to automorphism there can be only one such cycle. Figure~\ref{fig:hypotheticaldesargues} displays the Desargues graph where the 6-fold rotation preserves the outer cycle $X$, generated by the invertible element $g\in C$. The blue arcs leaving this cycle correspond to the other generator $h\in C$. Since $\langle C \rangle =M$ each vertex of the graph has to be reachable by a directed path from $e\in X$. By the rotation symmetry this implies that both inner (white) vertices must be sinks. Since $g$ is invertible, no element can have indegree larger than one with respect to $g$, both inner vertices are blue sinks and have a blue loop. In particular, $h^2=h^3$.
Since the neighbors of $X$ already have blue outdegree, in order to reach the inner cycle $Y$, the other arc leaving the neighbors of $X$ has to be black. Since $g$ is invertible, these edges must be black digons. Thus the arcs on $Y$ the inner cycle must all be blue, in particular $Y$ is directed and consists of six elements $\{hg,hgh,hgh^2,hgh^3,hgh^4,hgh^5\}$. This contradicts $h^2=h^3$.

\begin{figure}[ht]
\begin{center}
\includegraphics[width = .4\textwidth]{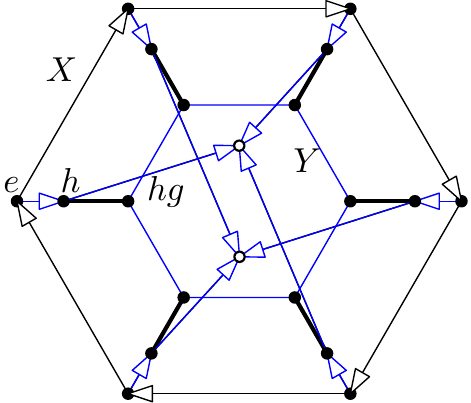}
\caption{A hypothetical representation of the Desargues graph with two generators one being invertible of order $6$.}
\label{fig:hypotheticaldesargues}
\end{center}
\end{figure}

Assume (e) holds and consider $C'$ is an inner cycle of length $n/\gcd(n,k) < n$. Then $\lambda_g = \alpha^{\mu k}$ for some $\mu \in \{-1,1\}$.
The outer neighbor of $e$ is $h$. We separate three cases and we are going to prove that none of them is possible.

\emph{Case 1:} $hg = h$. Since $M = \langle g,h \rangle$, then  $x g = x$ for all $x \notin C'$. As a consequence, the number of edges of $G(n,k)$ is at most $2n +  (n/\gcd(n,k)) < 3n = |E(G(n,k))|.$ 

\emph{Case 2:}  $hg \neq h$ and  $hg^2 = h$. Since $M = \langle g,h \rangle$, then  $x g^2 = x$ for all $x \notin C'$. As a consequence, the number of edges of $G(n,k)$ is at most $2n +  \frac{n}{\gcd(n,k)} + \frac{1}{2} \left(n - \frac{n}{\gcd(n,k)}\right) = \frac{5}{2} n + \frac{n}{2 \gcd(n,k)} < 3n = |E(G(n,k))|$.

\emph{Case 3:} $hg \neq h$ and $hg^2 \neq h$. We consider the closed walk $w = (h, hg, hg^2,\ldots,hg^{o(g)} = h)$ of length $o(g) < n$ in the  Cayley graph  ${\rm Cay}(V, \{g\})$ . Since the outer rim has length $n$, then the walk $w$ has to pass through (at least) two spokes connecting the outer rim with the same inner rim. However, since $\lambda_g = \alpha^{\pm k}$ is a color endomorphism (see Theorem~\ref{colorend}), this implies that these spokes are bioriented, but this contradicts that fact that the out-degree of every vertex in  ${\rm Cay}(V, \{g\})$ is $1$.
 \hfill \qed

Theorem~\ref{thm:mon2gen} is no longer true if we drop the assumption that  $M = \langle C \rangle$ (i.e., when $C$ is not a set of generators of $M$). Indeed, in Proposition~\ref{prop:Desargues}, we proved that the Desargues graph $G(10,3)$ is a monoid graph ${\rm Cay}(M,C)$ with $|C| = 2$. We are not aware of any further  generalized Petersen graphs with the same behavior. Indeed, we believe that this is the only exception, and that the generalized Petersen graphs that are monoid Cayley graphs with connection set of size $2$ are the ones described in Theorem~\ref{thm:mon2gen} and the Desargues graph.

\begin{conjecture}\label{conj:Desargues}
 The generalized Petersen graph $G(n,k)$ is a monoid graph ${\rm Cay}(M,C)$ with $|C| = 2$ if and only if one of the following holds: \begin{itemize} 
 \item[(a)] $(n,k)=(5,2)$ (Petersen graph),  \item[(b)] $(n,k)=(10,3)$ (Desargues graph), \item[(c)] $k^2 \equiv 1 \ ({\rm mod}\ n)$, \item[(d)] $k^2 \equiv \pm k\ ({\rm mod}\ n)$.
 \end{itemize}
\end{conjecture}

However, we are only able to prove it under the additional assumptions that $\gcd(n,k) = 1$ (Proposition~\ref{prop:mon2genprimos}), or that $n / \gcd(n,k)$ is odd (Proposition~\ref{prop:mon2gendentopar}).

\begin{proposition}
\label{prop:mon2genprimos}Let $1 \leq k < n/2$ such that $\gcd(n,k) = 1$. The generalized Petersen graph $G(n,k)$ is a monoid graph ${\rm Cay}(M,C)$ with $|C| = 2$ if and only if one of the following holds: \begin{itemize} \item[(a)] $(n,k)=(5,2)$ (Petersen graph), 
\item[(b)] $(n,k)=(10,3)$ (Desargues graph),  \item[(c)] $k^2 \equiv 1 \ ({\rm mod}\ n).$ \end{itemize}
\end{proposition}
\begin{proof}
$(\Longleftarrow)$ Follows from Proposition~\ref{prop:Petersen} (for $G(5,2)$), from Proposition~\ref{prop:Desargues} (for $G(10,3)$), and from Corollary~\ref{cor:petgroup} (for $k^2 \equiv 1 \ ({\rm mod}\ n)$).

\noindent  $(\Longrightarrow)$  Follows directly from Lemma~\ref{lm:technical} and observing that neither (d) nor (e) can  hold because $\gcd(n,k) = 1$. Indeed, if $k^2 \equiv \pm k \ ({\rm mod}\ n)$ and $\gcd(n,k) = 1$, then $n$ divides $k \mp 1$, but this cannot happen because  $k < n/2$. 
\end{proof}

\begin{proposition}
\label{prop:mon2gendentopar}Let $1 \leq k < n/2$ such that $n/\gcd(n,k)$ is odd. The generalized Petersen graph $G(n,k)$ is a monoid graph ${\rm Cay}(M,C)$ with $|C| = 2$ if and only if one of the following holds: \begin{itemize} \item[(a)] $(n,k)=(5,2)$ (Petersen graph), \item[(b)] $k^2 \equiv 1 \ ({\rm mod}\ n)$, \item[(c)] $k^2 \equiv \pm k\ ({\rm mod}\ n)$. \end{itemize}
\end{proposition}
\begin{proof}$(\Longleftarrow)$ Follows from Proposition~\ref{prop:Petersen} (for $G(5,2)$), from Corollary~\ref{cor:petgroup} (for $k^2 \equiv 1 \ ({\rm mod}\ n)$), and from Theorem~\ref{thm:Cay1}  (for $k^2 \equiv \pm k \ ({\rm mod}\ n)$).

$(\Longrightarrow)$ By Lemma~\ref{lm:technical} we just have to justify why (e) cannot hold. 
 So assume that the cycle $C'$ in the proof of Lemma~\ref{lm:technical} is an inner cycle, that $o(g) = n/\gcd(n,k)$ is odd and $o(g) < n$. Take $y \in V_O$ and let us prove  that $y g = y$. Indeed, consider the odd closed walk $w = (y, yg, yg^2,\ldots,yg^{o(g)} = y)$ of length $o(g) < n$ in the Cayley graph  ${\rm Cay}(V, \{g\})$. Since the outer cycle has length $n$ and $o(g) < n$ and is odd, then either $yg = y$ or the walk $w$ has to pass through (at least) two spokes connecting the outer cycle with the same inner cycle. Let us confirm that the latter is not possible. Since $\lambda_g$ is a color endomorphism (see Theorem~\ref{colorend}), then these spokes are bioriented, but this contradicts the fact that the out-degree of every vertex in  ${\rm Cay}(V, \{g\})$ is $1$ and the walk is odd. Thus $yg = y$. As a consequence, the outer rim is $(h,h^2,\ldots,h^{n+1} = h)$. However, in this case $h^{k+1} = g^{\pm 1} h$, which implies that $h^2 = (h g^{\pm 1}) h = h (g^{\pm 1} h) = h^{k+2}$, a contradiction.
\end{proof}

As a consequence of Proposition~\ref{prop:mon2gendentopar}, one gets that the dodecahedron $G(10,2)$ is not a monoid graph $\Cay(M,C)$ with $|C|=2$. However, as we have seen in Proposition~\ref{pr:dodeca} it is a monoid graph with $|C|=3$.

\section{Conclusions}
In the present paper we went beyond the classical symmetry properties for generalized Petersen graphs. We characterized cores and endomorphism transitive members of this class, see also Figure~\ref{fig:petersenplane}. This part of the research could be extended towards a deeper understanding of retracts of generalized Petersen graphs that are not cores. In particular, we believe that 
\begin{conjecture}
 The image of any endomorphism of a non bipartite generalized Petersen graph is a retract. 
\end{conjecture}

Note that this statement is not true for bipartite generalized Petersen graphs as shown by the M\"obius-Kantor graph $G(8,3)$ in Figure~\ref{fig:G83}. The white vertices are the image of an endomorphism but are not a retract because they do not induce an isometric subgraph.

\begin{figure}[ht]
\begin{center}
\includegraphics[width = .9\textwidth]{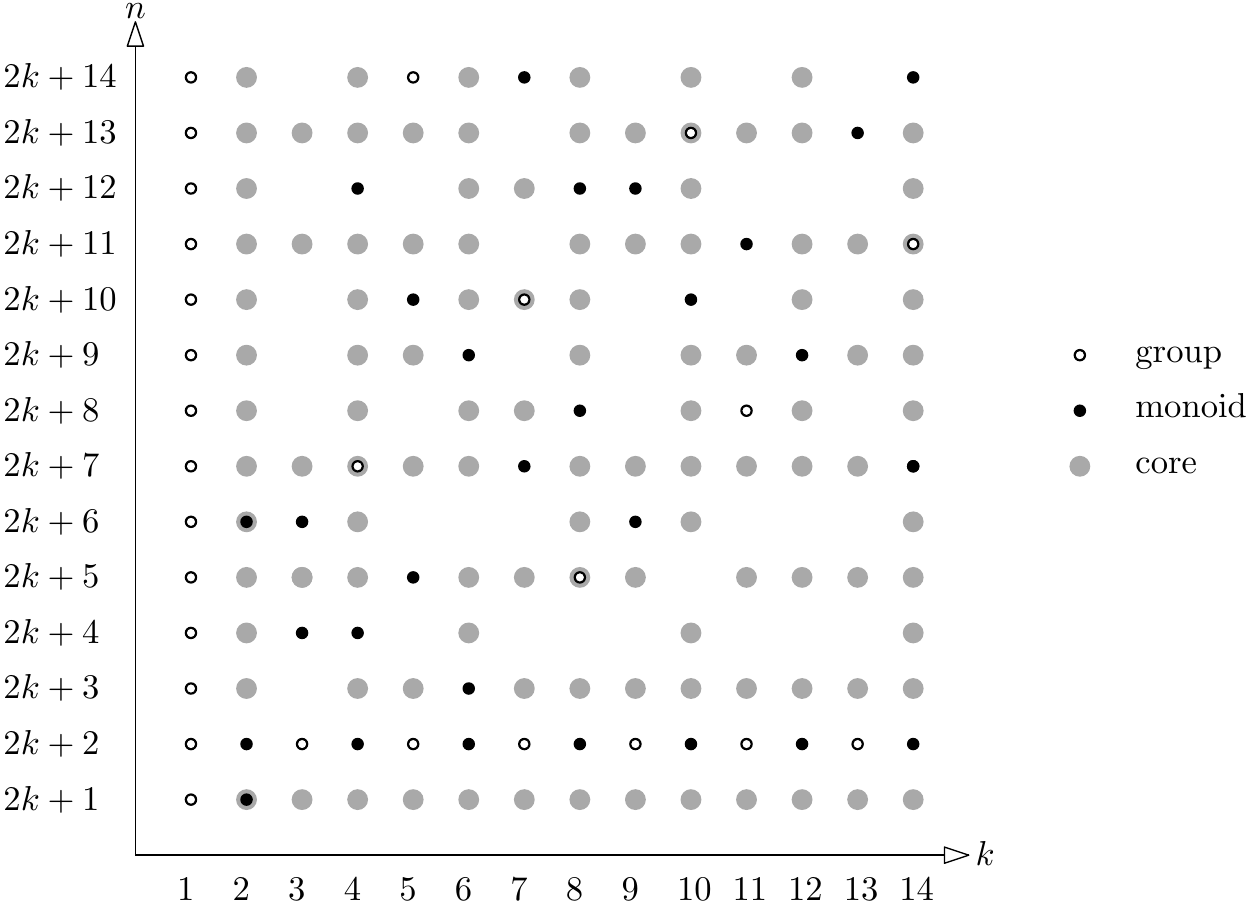}
\caption{Cores, group Cayley graphs, and known (non-group) monoid Cayley graphs in the generalized Petersen plane.}
\label{fig:petersenplane}
\end{center}
\end{figure}
 
Furthermore, we described large families of generalized Petersen graphs that are monoid graphs, see Figure~\ref{fig:petersenplane}. In particular, we characterized those generalized Petersen graphs that are monoid graphs with respect to a generating system of size $2$. We conjecture that the Desargues graph is the only other generalized Petersen graph that is monoid graphs with respect to a connection set of size $2$ (Conjecture~\ref{conj:Desargues}). However, we also exposed the dodecahedron being the only generalized Petersen graph we know of that is a monoid graph but only with respect to a connection set of size $3$. This proves that all graphs of Platonic solids are monoid graphs. In order to establish the same for Archimedean solids, one needs to find a monoid representation for the graph of the Icosidodecahedron. It remains open, whether all generalized Petersen graphs are monoid or semigroup graphs. The first cases are $G(7,2)$ and $G(7,3)$ fo which we know by Proposition~\ref{prop:mon2gendentopar} that they need a connection set of size $3$. Recall that while there are non-monoid graphs~\cite{Kna-21}, it is open whether all graphs are semigroup graphs. We dare:
\begin{conjecture}
All generalized Petersen graphs are semigroup graphs. 
\end{conjecture}

Given a loop-free semigroup representation of $G$, one gets easily one for the Kronecker double cover $G\times K_2$. However, comparing with~\cite{KP:2019} one finds that all Kronecker covers of graphs from the family in Theorem~\ref{thm:Cay1} that are Generalized Petersen graphs, fall into the same family. More generally, a \emph{covering map} from a graph $\hat{G}$ to a graph $G$ is a surjective graph homomorphism $\varphi:\hat{G}\to G$ such that for every vertex $v\in \hat{G}$,  $\varphi$ induces a one-to-one correspondence between edges incident to $v$ and edges incident to $\varphi(v)$. If there is a covering map from $\hat{G}$ to $G$, we say that $\hat{G}$ is a \emph{covering} of $G$. For instance both the dodecahedron as well as the Desargues graph are coverings of the Petersen graph. It is an interesting question what further properties are needed in order to lift a loop-free semigroup representation a graph to its covering. On the other hand, generalizing results of~\cite{KP:2019}, we wonder 
\begin{question}
Which generalized Petersen graphs are (non-trivial) coverings of  generalized Petersen graphs? 
\end{question}

\subsubsection*{Acknowledgments.}
We thank Ulrich Knauer for helpful comments on the manuscript. 
The second author was partially supported by the French \emph{Agence nationale de la recherche} through
project ANR-17-CE40-0015 and by the Spanish \emph{Ministerio de Econom\'ia, Industria y Competitividad}
through grant RYC-2017-22701. Both authors were partially supported by the Spanish
MICINN through grant PID2019-104844GB-I00 and by the Universidad de La Laguna MASCA and MACACO
projects.

\bibliography{lit}
\bibliographystyle{my-siam}
\end{document}